\documentclass[11pt]{amsart}

\usepackage[utf8]{inputenc}
\usepackage[T1]{fontenc}

\usepackage{amsmath,amssymb,amsfonts,textcomp,amsthm,xifthen,graphicx,color,pgfplots}
\usepackage{ifthen}

\usepackage{enumerate}
\usepackage{enumitem}

\usepackage{fullpage}

\usepackage[utf8]{inputenc}

\usepackage[pdftex,
            pdfauthor={Thomas F\"uhrer},
            pdftitle={Superconvergent DPG methods for second order elliptic problems},
            ]{hyperref}

\usepgfplotslibrary{external}

\usepackage{float} 
\usepackage{array}

\usepackage{multirow}

\usepackage{etoolbox}
\newcounter{rowcntr}[table]
\renewcommand{\therowcntr}{\alph{rowcntr})}

\newcolumntype{N}{>{\refstepcounter{rowcntr}\therowcntr}c}

\AtBeginEnvironment{tabular}{\setcounter{rowcntr}{0}}

\newtheorem{theorem}{Theorem}
\newtheorem{lemma}[theorem]{Lemma}

\newtheorem{proposition}[theorem]{Proposition}
\newtheorem{remark}[theorem]{Remark}

\newcommand{\ip}[2]{(#1\hspace*{.5mm},#2)}
\newcommand{\dual}[2]{\langle#1\hspace*{.5mm},#2\rangle}

\newcommand{\norm}[3][]{#1\|#2#1\|_{#3}}

\newcommand{\snorm}[2]{|#1|_{#2}}

\newcommand{\diam}{\mathrm{diam}}

\newcommand{\wat}{\widehat}

\def\div{{\rm div\,}}

\newcommand{\Hdivset}[1]{\boldsymbol{H}(\div;#1)}

\newcommand{\set}[2]{\big\{#1\,:\,#2\big\}}

\newcommand{\pwnabla}{\nabla_\TT}
\newcommand{\pwdiv}{\div_\TT}

\newcommand{\VnormQopt}[1]{\norm{#1}{\mathrm{V,qopt}}}
\newcommand{\VnormStd}[1]{\norm{#1}{\mathrm{V},1}}
\newcommand{\VnormSimple}[1]{\norm{#1}{\mathrm{V},2}}

\newcommand{\RT}{\ensuremath{\mathcal{RT}}}

\newcommand{\R}{\ensuremath{\mathbb{R}}}
\newcommand{\N}{\ensuremath{\mathbb{N}}}
\newcommand{\HH}{\ensuremath{{\boldsymbol{H}}}}

\newcommand{\LL}{\ensuremath{\boldsymbol{L}}}
\newcommand{\g}{\ensuremath{\boldsymbol{g}}}

\newcommand{\vv}{\ensuremath{\boldsymbol{v}}}
\newcommand{\ww}{\ensuremath{\boldsymbol{w}}}
\newcommand{\TT}{\ensuremath{\mathcal{T}}}

\newcommand{\cS}{\ensuremath{\mathcal{S}}}

\newcommand{\PP}{\ensuremath{\mathcal{P}}}
\newcommand{\OO}{\ensuremath{\mathcal{O}}}

\newcommand{\normal}{\ensuremath{{\boldsymbol{n}}}}

\newcommand{\ff}{\ensuremath{\boldsymbol{f}}}
\newcommand{\CC}{\ensuremath{\boldsymbol{C}}}
\newcommand{\bbeta}{\ensuremath{\boldsymbol{\beta}}}

\newcommand{\ssigma}{{\boldsymbol\sigma}}
\newcommand{\ttau}{{\boldsymbol\tau}}
\newcommand{\llambda}{{\boldsymbol\lambda}}
\newcommand{\cchi}{{\boldsymbol\chi}}

\newcommand{\uu}{\boldsymbol{u}}
\newcommand{\eeps}{\boldsymbol{\varepsilon}}

\newcounter{constantsnumber}
\def\setc#1{
  \ifthenelse{\equal{#1}{poinc}}{C_{\rm edge}}{ 
   \refstepcounter{constantsnumber}
   \label{const#1}C_{\theconstantsnumber}}}

\def\c#1{
  \ifthenelse{\equal{#1}{poinc}}{C_{\rm edge}}{ 
    C_{\ref{const#1}}}}

\newcommand{\xx}{{\boldsymbol{x}}}

\begin{document}

\title[Superconvergent DPG methods for second order elliptic problems]{Superconvergent DPG methods for second order elliptic problems}
\date{\today}

\author{Thomas F\"{u}hrer}
\address{Facultad de Matem\'{a}ticas, Pontificia Universidad Cat\'{o}lica de Chile, Santiago, Chile}
\email{tofuhrer@mat.uc.cl}

\thanks{{\bf Acknowledgment.} 
This work was supported by FONDECYT project 11170050.}

\begin{abstract}
We consider DPG methods with optimal test functions and broken test spaces based on ultra-weak formulations of general
second order elliptic problems.
Under some assumptions on the regularity of solutions of the model problem and its adjoint, superconvergence for the
scalar field variable is achieved by either increasing the polynomial degree in the corresponding approximation space by
one or by a local postprocessing.
We provide a uniform analysis that allows to treat different test norms.
Particularly, we show that in the presence of convection only the quasi-optimal test norm leads to higher convergence
rates, whereas other norms considered do not.
Moreover, we also prove that our DPG method delivers the best $L^2$ approximation of the scalar field variable
up to higher order terms, which is the first theoretical explanation of an observation made previously by different authors.
Numerical studies that support our theoretical findings are presented.
\end{abstract}
\maketitle

\section{Introduction}
In this work we investigate convergence rates of DPG methods based on an ultra-weak formulation of second order elliptic
problems stated in the form of the general first-order system 
\begin{subequations}\label{eq:model}
\begin{alignat}{2}
  \nabla u -\bbeta u + \CC\ssigma &= \CC\ff &\quad&\text{in }\Omega, \label{eq:model:a} \\
  \div\ssigma + \gamma u &= f &\quad&\text{in }\Omega, \label{eq:model:b}\\
  u &=0 &\quad&\text{on }\Gamma:=\partial \Omega,
\end{alignat}
\end{subequations}
where $\Omega\subseteq \R^d$, $d\geq2$, is a polyhedral domain and 
$\CC\in L^\infty(\Omega)^{d\times d}$ denotes a symmetric, uniformly positive definite matrix valued function, 
$\bbeta \in L^\infty(\Omega)^d$, $\gamma \in L^\infty(\Omega)$.
Throughout we suppose that the coefficients additionally satisfy 
\begin{align}\label{eq:ass:coeff}
  L^\infty(\Omega) \ni \tfrac12 \div(\CC^{-1}\bbeta)+\gamma \geq 0,
\end{align}
which implies that for $f\in L^2(\Omega)$, $\ff\in\LL^2(\Omega):=L^2(\Omega)^d$ our model problem~\eqref{eq:model} 
admits a unique solution $(u,\ssigma)$ with $u\in H_0^1(\Omega)$, $\ssigma\in \Hdivset\Omega := \set{\ttau\in\LL^2(\Omega)}{\div\ttau\in
L^2(\Omega)}$.

In this work we consider DPG methods with optimal test functions and broken test spaces,
which have been introduced by Demkowicz \& Gopalakrishnan, see~\cite{partI,partII} and also~\cite{partIII,partIV}. 
For a unified stability analysis which also covers our model problem we refer to~\cite{breakSpace}.
We analyze ultra-weak formulations of~\eqref{eq:model}, which
are obtained by multiplying with locally supported functions and integration
by parts, see, e.g.,~\cite{DemkowiczG_11_ADM} for a Poisson model problem.
On the one hand, this has the advantage that the field variables can be sought in $L^2(\Omega)$, since no derivative
operator is applied to these unknowns after integration by parts.
On the other hand, this requires the introduction of trace variables $\wat u$, $\wat\sigma$ that live on the skeleton
(these unknowns impose weak continuity conditions).
However, as analyzed in the recent work~\cite{PolyDPG} the use of ultra-weak formulations also 
allows to define conforming finite element spaces on polygonal meshes.

The motivation of this work is to analyze superconvergence properties for approximations of the scalar field variable
$u$ that have been observed in our recent work~\cite{SupconvDPG} for a simple reaction-diffusion problem, where
$\CC$ is the identity matrix, $\bbeta = 0$, $\gamma=1$, and $\ff=0$.
Here we generalize and extend~\cite{SupconvDPG} to the model problem~\eqref{eq:model} and introduce new ideas that
allow to treat different test norms.
As in~\cite{SupconvDPG}, the proofs rely on duality arguments and regularity theory for elliptic PDEs. 
Such arguments are common when proving higher convergence rates, e.g., the classical Aubin-Nitsche trick, or more
recently in variants of DG methods, e.g.~\cite{CockburnDG_superconvergence}.

Let us also mention the recent works~\cite{DPGstar,GoalOrientedDPG} that deal with dual
problems in the context of DPG methods (the DPG$^*$ method and goal-oriented problems).
Particularly, we point out the reference~\cite{BoumaGH_DPGconvRates}. There the authors consider a primal DPG method
(without the first-order reformulation) for the Poisson problem and analyze convergence rates 
(with reduced degrees in test spaces). Moreover, they develop duality arguments and prove that the 
error in the primal variable $u$ converges at a higher rate when measured in a weaker norm.

\subsection{Summary of results}
We seek approximations $u_h\in\PP^p(\TT)$, $\ssigma_h\in\PP^p(\TT)^d$ of the 
field variables $u$, $\ssigma$, where $\TT$ is a mesh of simplices and $\PP^p(\TT)$ denotes the space of $\TT$-piecewise
polynomials of degree less than or equal to $p\in\N_0$, and approximations $\wat u_h$, $\wat\sigma_h$ of the traces $\wat u$, $\wat\sigma$ in
spaces that will be defined later on.
For sufficient regular solutions basic a priori analysis arguments give the estimate
\begin{align*}
  \norm{\uu-\uu_h}U \simeq \norm{u-u_h}{} + \norm{\ssigma-\ssigma_h}{} + \norm{(\wat u-\wat u_h,\wat\sigma-\wat\sigma_h)}{\cS} 
  = \OO(h^{p+1}),
\end{align*}
where $\norm{\cdot}{}$ denotes the $L^2(\Omega)$ norm and $\norm\cdot{\cS}$ is some appropriate norm for the traces.
This estimate is optimal, since we seek approximations of $u$ and $\ssigma$ in polynomial spaces of the same order 
and their errors are measured in $L^2(\Omega)$ norms.
Nevertheless, it is unsatisfactory to some extent. Consider $\CC$ the identity, $\bbeta=0$, $\ff=0$
in~\eqref{eq:model}. Then, $\ssigma=\nabla u$ and we seek approximations of $u$ and its gradient $\ssigma$ in 
polynomial spaces of the same order, which seems to be suboptimal.
Fortunately, we can prove at least two possibilities to achieve higher convergence rates 
under some assumptions on the regularity of solutions of~\eqref{eq:model} and its adjoint problem:
\begin{itemize}
  \item \emph{Augmenting the trial space:} Instead of seeking approximations $u_h\in\PP^p(\TT)$ we seek approximations
    $u_h^+\in\PP^{p+1}(\TT)$ and show that
    \begin{align*}
      \norm{u-u_h^+}{} = \OO(h^{p+2}).
    \end{align*}
  \item \emph{Postprocessing:} We use a postprocessing technique that goes at least back to~\cite{stenbergPostPr} to obtain an
    approximation $\widetilde u_h\in\PP^{p+1}(\TT)$ and prove that
    \begin{align*}
      \norm{u-\widetilde u_h}{} = \OO(h^{p+2}).
    \end{align*}
\end{itemize}
Based on similar techniques we also provide a proof of the following:
\begin{itemize}
  \item \textit{DPG for ultra-weak formulations delivers the $L^2(\Omega)$ best approximation up to a higher order term}, 
    i.e., for the approximation
    $u_h\in\PP^p(\TT)$ it holds
    \begin{align*}
      \norm{u-u_h}{} \leq \norm{u-\Pi^p u}{} + \OO(h^{p+2}),
    \end{align*}
    where $\Pi^p$ denotes the $L^2(\Omega)$ projection to $\PP^p(\TT)$.
\end{itemize}
The latter observation is quite interesting, because it shows that even though we do not aim for higher convergence
rates (by increasing the polynomial degree in the trial space or by postprocessing) we get highly accurate
approximations.
We stress that this result has been observed in various numerical experiments, particularly also for more complex model
problems like Stokes~\cite{DPGStokes}, but up to now a rigorous proof has not been given.

If $\bbeta=0$, we show that these results hold true when using different test norms (one of them is the so-called quasi-optimal
test norm). 
Surprisingly (at this point), 
for $\bbeta\neq0$ the results are only valid if the quasi-optimal test norm is used, although all test
norms under consideration are equivalent. This is also observed in our numerical studies.

\subsection{Basic ideas}
For the proofs of the main results, we develop duality arguments and show approximation results 
(Lemma~\ref{lem:w:rep} and Lemma~\ref{lem:est}).
To get the essential idea, consider the abstract formulation: Find $\uu\in U$ such that
\begin{align*}
  b(\uu,\vv) = F(\vv) \quad\text{for all }\vv\in V,
\end{align*}
where $U$ denotes the trial space and $V$ the test space. With the trial-to-test operator $\Theta : U\to V$,
\begin{align*}
  \ip{\Theta\ww}\vv_V = b(\ww,\vv) \quad\text{for all }\vv\in V,
\end{align*}
the ideal DPG method reads: Find $\uu_h \in U_h\subset U$ such that
\begin{align*}
  b(\uu_h,\Theta\ww_h) = F(\Theta\ww_h) \quad\text{for all }\ww_h\in U_h.
\end{align*}
Then, we solve a dual problem:
For some given $g\in L^2(\Omega)$, we determine
$\vv\in V$ and $\ww = \Theta^{-1}\vv\in U$, both unique, and employ Galerkin orthogonality to obtain
\begin{align*}
  \ip{u-u_h}{g} = b(\uu-\uu_h,\vv) = b(\uu-\uu_h,\Theta\ww) = b(\uu-\uu_h,\Theta(\ww-\ww_h)) 
  \lesssim \norm{\uu-\uu_h}U\norm{\ww-\ww_h}U.
\end{align*}
for arbitrary $\ww_h\in U_h$.

For the case, where we want to show that the approximation $u_h\in\PP^p(\TT)$ is nearly the
$L^2(\Omega)$ best approximation, we have $g = \Pi^p(u-u_h)$.
Therefore,
\begin{align*}
  \norm{g}{}^2 = \ip{u-u_h}{g} \lesssim \norm{\uu-\uu_h}U\norm{\ww-\ww_h}U \lesssim \norm{\uu-\uu_h}U h
  \norm{g}{}.
\end{align*}
The latter estimate is what we have to show. Suppose that it holds. With the estimate for $\norm{\uu-\uu_h}U$ from
above, it is straightforward to see that
\begin{align*}
  \norm{u-u_h}{} \leq \norm{u-\Pi^p u}{} + \norm{\Pi^p u-u_h}{} =  \norm{u-\Pi^p u}{} + \norm{g}{} 
  = \norm{u-\Pi^p u}{} + \OO(h^{p+2}).
\end{align*}

Let us come back to the essential estimate
\begin{align*}
  \norm{\ww-\ww_h}U \lesssim h \norm{g}{}.
\end{align*}
It holds if we would know that 
the higher derivatives of $\ww$ exist (in some sense) and can be bounded by the norm of $g$, so that,
formally,
\begin{align*}
  \norm{\ww-\ww_h}U \lesssim h \norm{D^\mathrm{higher}\ww}{} \lesssim h \norm{g}{}
\end{align*}
by some standard arguments. In our case we have that $\vv \in H_0^1(\Omega)\times \Hdivset\Omega \subset V$ 
is the solution to the adjoint problem of~\eqref{eq:model} and under some assumptions has the higher regularity $\vv
\in H^2(\Omega)\times \HH^1(\TT)\cap\Hdivset\Omega$, where $\HH^1(\TT)$ denotes $\TT$-piecewise Sobolev functions.
Recall that $\ww = \Theta^{-1}\vv$. One difficulty is that the inverse of the trial-to-test operator does not map regular
functions back to regular functions. 
However, it turns out (Lemma~\ref{lem:w:rep}) that $\ww$ can be written as
\begin{align*}
  \ww = (g,0,0,0) + \widetilde\ww + \ww^\star,
\end{align*}
where components of $\widetilde\ww \in U$ are connected to the dual solution $\vv$, which is sufficient regular
and $\ww^\star$ is the solution of the (primal) problem~\eqref{eq:model} with data $f$ and $\ff$ depending on the dual
solution $\vv$ so that $\ww^\star$ has sufficient regularity as well.
Let us point out that this idea used in the proofs is new and allows to treat different test norms.
In~\cite{SupconvDPG}, which deals with a simple reaction-diffusion problem and one specific test norm only, 
the representation of $\ww$ is obtained by integration by parts using the dual solution $\vv$ and it is not clear if
that approach can be generalized to the present setting.
Here, in the general case we have to consider the regularity of the dual solution $\vv$ and the
regularity of the solution $\ww^\star$ of the primal problem.
For the proofs it is also necessary that $g$ is a function in the 
finite element space, so that we can choose $\ww_h = (g,0,0,0) + \overline\ww_h$, where $\overline\ww_h$ is the best
approximation of $\widetilde\ww+\ww^\star$. Then, we show that the above estimates hold true.

Let us note that $\Theta$ is defined through the inner product in the test space. Thus,
the representation of $\ww = \Theta^{-1}\vv$ from above strongly depends on the choice of the test norm 
and has to be analyzed for each norm individually (this is done in Lemma~\ref{lem:w:rep}).

Moreover, the ideas so far dealt with the ideal DPG method.
In this paper we work out all results for the practical DPG method under standard assumptions, i.e, the existence of Fortin operators.
This implies that we have to deal with additional discretization errors.

\subsection{Outline}
The remainder of the paper is organized as follows:
Section~\ref{sec:main} introduces basic notations, states the assumptions, and presents the main results
(Theorem~\ref{thm:L2}--\ref{thm:postproc}).
The proofs of these theorems are postponed to Section~\ref{sec:proof}, which also includes 
an a priori convergence estimate (Theorem~\ref{thm:approx}) and the important auxiliary results
Lemma~\ref{lem:w:rep},~\ref{lem:est}.
In Section~\ref{sec:ex} we present two numerical experiments.
The final Section~\ref{sec:remarks} concludes this work with some remarks.

\section{Main results}\label{sec:main}

\subsection{Notation}
We make use of the notation $\lesssim$, i.e., $A\lesssim B$ means that there exists a constant $C>0$, which is
independent of relevant quantities, such that $A\leq C B$.
Moreover, $A\simeq B$ means that both directions hold, i.e., $A\lesssim B$ and $B\lesssim A$.

\subsection{Mesh}
Let $\TT$ denote a regular mesh of $\Omega$ consisting of simplices $T$ and let $\cS := \set{\partial T}{T\in\TT}$
denote the skeleton. We suppose that $\TT$ is shape-regular, i.e., there exists a constant $\kappa_\TT>0$ such that
\begin{align*}
  \max_{T\in\TT} \frac{\diam(T)^d}{|T|} \leq \kappa_\TT,
\end{align*}
where $|T|$ denotes the volume measure of $T\in\TT$. As usual $h := h_\TT := \max_{T\in\TT} \diam(T)$ denotes the
mesh-size.

\subsection{Ultra-weak formulation}
Before we derive the ultra-weak formulation of~\eqref{eq:model} in this section, we introduce some notation.
Let $T\in\TT$. 
We denote by $\ip\cdot\cdot_T$ the $L^2(T)$ scalar product and with $\norm\cdot{T}$ the induced norm.
On boundaries $\partial T$, the $L^2(\partial T)$ scalar product is denoted by $\dual\cdot\cdot_{\partial T}$ and
extended to the duality between the spaces $H^{1/2}(\partial T)$ and $H^{-1/2}(\partial T)$.
Furthermore, we define the piecewise trace operators 
\begin{alignat*}{2}
  \gamma_{0,\cS} &: H^1(\Omega) \to \prod_{T\in\TT} H^{1/2}(\partial T), \quad &(\gamma_{0,\cS} v)|_{\partial T} =&
  v|_{\partial T}, \\
  \gamma_{\normal,\cS} &: \Hdivset\Omega \to \prod_{T\in\TT} H^{-1/2}(\partial T), 
  \quad &(\gamma_{\normal,\cS} \ttau)|_{\partial T} =& \ttau\cdot\normal_T|_{\partial T},
\end{alignat*}
where $\normal_T$ denotes the normal on $\partial T$ pointing from $T$ to its complement.
With these operators we define the trace spaces
\begin{align*}
  H_0^{1/2}(\cS) := \gamma_{0,\cS}(H_0^1(\Omega)), \quad\text{and}\quad
  H^{-1/2}(\cS) := \gamma_{\normal,\cS}(\Hdivset\Omega).
\end{align*}
These Hilbert spaces are equipped with minimum energy extension norms
\begin{align*}
  \norm{\wat u}{1/2,\cS} := \inf\set{\norm{u}{H^1(\Omega)}}{\gamma_{0,\cS}u = \wat u}, \quad
  \norm{\wat\sigma}{-1/2,\cS} := \inf\set{\norm{\ssigma}{\Hdivset\Omega}}{\gamma_{\normal,\cS}\ssigma = \wat\sigma}.
\end{align*}

We use the broken test spaces 
\begin{align*}
  H^1(\TT) &:= \set{v\in L^2(\Omega)}{v|_T \in H^1(T) \text{ for all }T\in\TT}, \\
  \Hdivset\TT &:= \set{\ttau\in\LL^2(\Omega)}{\ttau|_T \in\Hdivset{T} \text{ for all }T\in\TT}
\end{align*}
and define the piecewise differential operators $\pwnabla : H^1(\TT) \to \LL^2(\Omega)$, $\pwdiv : \Hdivset\TT \to
L^2(\Omega)$ on each $T\in\TT$ by
\begin{align*}
  \pwnabla v|_T := \nabla(v|_T), \quad \pwdiv\ttau|_T := \div(\ttau|_T).
\end{align*}
Moreover, we define the dualities
\begin{align*}
  \dual{\wat u}{\ttau\cdot\normal}_\cS := 
  \sum_{T\in\TT} \dual{\wat u|_{\partial T}}{\ttau\cdot\normal_T|_{\partial T}}_{\partial T},
  \quad
  \dual{\wat\sigma}v_\cS := \sum_{T\in\TT} \dual{\wat\sigma|_{\partial T}}{v|_{\partial T}}_{\partial T}
\end{align*}
for all $\wat u\in H_0^{1/2}(\cS)$, $\ttau\in\Hdivset\TT$, $\wat\sigma\in H^{-1/2}(\cS)$, $v\in H^1(\TT)$.
These dualities measure the jumps of $\vv =(v,\ttau)\in H^1(\TT)\times \Hdivset\TT$, i.e.,
\begin{subequations}\label{eq:jumps}
\begin{alignat}{2}
  v\in H_0^1(\Omega) &\Longleftrightarrow \dual{\widehat \sigma}v_\cS = 0 &\quad&\text{for all } \widehat\sigma \in
  H^{-1/2}(\cS), \\
  \ttau\in \Hdivset\Omega &\Longleftrightarrow \dual{\widehat u}{\ttau\cdot\normal}_\cS = 0 
  &\quad&\text{for all }\widehat u\in H_0^{1/2}(\cS),
\end{alignat}
\end{subequations}
see, e.g.,~\cite[Theorem~2.3]{breakSpace}.

The ultra-weak formulation is then derived from~\eqref{eq:model} by testing~\eqref{eq:model:a} with
$\ttau\in\Hdivset\TT$,~\eqref{eq:model:b} with $v\in H^1(\TT)$, and piecewise integration by parts,
i.e.,
\begin{align*}
  -\ip{u}{\pwdiv\ttau} + \dual{\gamma_{0,\cS}u}{\ttau\cdot\normal}_\cS
  -\ip{\bbeta u}\ttau + \ip{\CC\ssigma}\ttau &= \ip{\CC\ff}\ttau, \\
  -\ip{\ssigma}{\pwnabla v} + \dual{\gamma_{\normal,\cS}\ssigma}v_\cS
  +\ip{\gamma u}v &= \ip{f}v.
\end{align*}
Here, $\ip\cdot\cdot := \ip\cdot\cdot_\Omega$ is the $L^2(\Omega)$ scalar product with norm $\norm\cdot{}$.
Set
\begin{align*}
  U := L^2(\Omega)\times \LL^2(\Omega) \times H_0^{1/2}(\cS) \times H^{-1/2}(\cS), \quad V := H^1(\TT)\times \Hdivset\TT
\end{align*}
and define $F : V\to \R$ and $b : U\times V\to\R$ by
\begin{align*}
  F(\vv) &:= \ip{f}v + \ip{\ff}{\CC\ttau}, \\
  b(\uu,\vv) &:= \ip{u}{-\pwdiv\ttau-\bbeta\cdot\ttau+\gamma v} + \ip{\ssigma}{\CC\ttau-\pwnabla v}
  + \dual{\wat u}{\ttau\cdot\normal}_\cS + \dual{\wat\sigma}v_\cS
\end{align*}
for all $\uu = (u,\ssigma,\wat u,\wat\sigma)\in U$, $\vv= (v,\ttau)\in V$.
The ultra-weak formulation then reads: Find $\uu\in U$ such that
\begin{align}\label{eq:uwf}
  b(\uu,\vv) = F(\vv) \quad\text{for all }\vv\in V.
\end{align}

\subsection{DPG method and approximation}
In $U$ we use the canonical norm,
\begin{align*}
  \norm{\uu}U^2 := \norm{u}{}^2 + \norm{\ssigma}{}^2 + 
  \norm{\wat u}{1/2,\cS}^2 + \norm{\wat\sigma}{-1/2,\cS}^2 \quad\text{for }\uu = (u,\ssigma,\wat u,\wat\sigma)\in U.
\end{align*}
For the test space $V$ we define the three different norms
\begin{subequations}\label{eq:defVnorms}
\begin{align}
  \VnormQopt\vv^2 &:= \norm{-\pwdiv\ttau-\bbeta\cdot\ttau+\gamma v}{}^2 
  + \norm{\CC^{1/2}\ttau-\CC^{-1/2}\pwnabla v}{}^2 + \norm{\CC^{1/2}\ttau}{}^2 
  + \norm{v}{}^2, \\
  \VnormStd\vv^2 &:= \norm{\CC^{-1/2}\pwnabla v}{}^2 + \norm{v}{}^2 
  + \norm{\pwdiv\ttau}{}^2 + \norm{\CC^{1/2}\ttau}{}^2,\\
  \VnormSimple\vv^2 &:= \norm{\pwnabla v}{}^2 + \norm{v}{}^2 + \norm{\pwdiv\ttau}{}^2 + \norm{\ttau}{}^2
\end{align}
\end{subequations}
for $\vv=(v,\ttau)\in V$ and denote by $\ip\cdot\cdot_{V,\star}$ the corresponding scalar products.
Note that all norms in~\eqref{eq:defVnorms} are equivalent with equivalence constants depending on the coefficients
$\CC$, $\bbeta$, $\gamma$.
However, our main results hold for the quasi-optimal
test norm $\VnormQopt\cdot$ under mild assumptions on the coefficient $\bbeta$, 
whereas they hold for $\VnormStd\cdot$, $\VnormSimple\cdot$ only if $\bbeta=0$, i.e. for symmetric problems.

We stress that $b : U\times V\to \R$ is a bounded bilinear form and satisfies the $\inf$--$\sup$ conditions with mesh
independent constant. This can be proved with the theory developed in~\cite{breakSpace}. 
For our model problem we explicitly refer to~\cite[Example~3.7]{breakSpace} for the details. 
There it is assumed that $\div(\CC^{-1}\bbeta)=0$ and $\gamma\geq 0$. We note that their analysis can also be done with
our more general assumption~\eqref{eq:ass:coeff}.

The DPG method, seeks an approximation $\uu_h\in U_h\subset U$ of the solution $\uu\in U$ using the optimal test space
$\Theta(U_h)$, where $\Theta : U\to V$ is defined by
\begin{align}\label{def:ttt}
  \ip{\Theta\ww}\vv_V = b(\ww,\vv) \quad\text{for all } \ww\in U,\vv\in V.
\end{align}
Then, $\uu_h\in U_h$ is the solution of
\begin{align*}
  b(\uu_h,\vv_h) = F(\vv_h) \quad\text{for all } \vv_h\in \Theta(U_h).
\end{align*}
An essential feature of DPG is that $\inf$--$\sup$ stability directly transfers to the discrete problem. However, in
practice we replace $\Theta$ by a discrete version $\Theta_h : U_h \to V_h\subset V$ defined by
\begin{align}\label{def:ttt:discrete}
  \ip{\Theta_h\ww_h}{\vv_h}_V = b(\ww_h,\vv_h) \quad\text{for all }\ww_h\in U_h,\vv_h\in V_h.
\end{align}
Then, the \emph{practical DPG method} reads: Find $\uu_h\in U_h$ such that
\begin{align}\label{eq:practicalDPG}
  b(\uu_h,\Theta_h\ww_h) = F(\Theta_h\ww_h) \quad\text{for all }\ww_h\in U_h.
\end{align}

In this work we deal with the piecewise polynomial trial spaces
\begin{align*}
  U_{hp} &:= \PP^p(\TT)\times \PP^p(\TT)^d \times \PP_{c,0}^{p+1}(\cS) \times \PP^p(\cS), \\
  U_{hp}^+ &:= \PP^{p+1}(\TT)\times \PP^p(\TT)^d \times \PP_{c,0}^{p+1}(\cS) \times \PP^p(\cS)
\end{align*}
and the piecewise polynomial test spaces
\begin{align*}
  V_{hk} := \PP^{k_1}(\TT) \times \PP^{k_2}(\TT).
\end{align*}
Here, we set 
\begin{align*}
  \PP^p(T) &:= \set{v\in L^2(T)}{v \text{ is polynomial of degree } \leq p}, \\
  \PP^p(\TT) &:= \set{v\in L^2(\Omega)}{v|_T \in \PP^p(T),\,T\in\TT}, \quad \PP_{c,0}^{p+1}(\TT) := \PP^{p+1}(\TT)\cap
  H_0^1(\Omega)  \\
  \PP_{c,0}^{p+1}(\cS) &:= \gamma_{0,\cS}\left( \PP_{c,0}^{p+1}(\TT)\right), \quad
  \PP^p(\cS) := \gamma_{\normal,\cS}\left( \RT^p(\TT)\right),
\end{align*}
where $\RT^p(\TT) = \set{\ttau\in\Hdivset\Omega}{\ttau|_T(\xx) = \boldsymbol{a}+\xx b,\, \boldsymbol{a}\in\PP^p(T)^d,\, b\in
\widetilde\PP^p(T), \, T\in\TT}$ is the space of Raviart-Thomas functions (here $\widetilde\PP^p(T)$ denotes the space of
homogeneous polynomials of degree $p$).

We also use the space $C^1(\TT) := \set{v\in L^\infty(\Omega)}{v|_T \in C^1(\overline T), \, T\in\TT}$.

\subsection{Fortin operators}\label{sec:fortin}
It is well-known, see e.g.~\cite{practicalDPG}, that~\eqref{eq:practicalDPG} satisfies $\inf$--$\sup$ conditions (and
therefore admits a unique solution) if there exists a Fortin operator $\Pi_F : V\to V_h$ such that
\begin{align}\label{eq:fortin}
  \norm{\Pi_F\vv}V\leq C_F \norm{\vv}V \quad\text{and}\quad 
  b(\uu_h,\vv) = b(\uu_h,\Pi_F\vv) \quad\text{for all }\vv\in V, \uu_h\in U_h.
\end{align}
Throughout, we suppose that a Fortin operator exists for the discrete polynomial trial and test spaces under
consideration and that $C_F$ depends only on $\CC$, $\bbeta$, $\gamma$, $p\in\N_0$, and the shape-regularity of $\TT$.
Let us note that for general coefficients $\CC$, $\bbeta$, $\gamma$ the existence of such operators is
not known, except for some special cases, i.e., the Poisson model problem where $\CC$ is the identity and
$\bbeta=0=\gamma$.
Fortin operators for the latter problem on simplicial meshes have been constructed and analyzed in~\cite{practicalDPG}.
We refer also to~\cite{constrFortin} for the construction and analysis of Fortin operators for second order problems.

Supposing the existence of an Fortin operator, i.e.,~\eqref{eq:fortin}, we have:
\begin{proposition}\label{prop:dpg}
  Problems~\eqref{eq:uwf},~\eqref{eq:practicalDPG} admit 
  unique solutions $\uu=(u,\ssigma,\wat u,\wat\sigma)\in U$, $\uu_h\in U_h$ and
  \begin{align*}
    \norm{\uu-\uu_h}{U} \leq C_\mathrm{opt} \min_{\ww_h\in U_h} \norm{\uu-\ww_h}U.
  \end{align*}
  The constant $C_\mathrm{opt}>0$ depends only on $\Omega$, $\CC$, $\bbeta$, $\gamma$, $p\in\N_0$, and shape-regularity
  of $\TT$.
\qed
\end{proposition}

\subsection{Adjoint problem and regularity assumptions}\label{sec:model}
We define the adjoint problem of~\eqref{eq:model} as
\begin{subequations}\label{eq:adjoint}
\begin{alignat}{2}
  -\div\ttau -\bbeta\cdot\ttau + \gamma v &= g &\quad&\text{in }\Omega, \\
  \CC\ttau-\nabla v &= \CC\g &\quad&\text{in }\Omega, \\
  u &=0 &\quad&\text{on }\Gamma.
\end{alignat}
\end{subequations}
Again, this problem admits a unique solution $(v,\ttau)\in H_0^1(\Omega) \times \Hdivset\Omega$ for $g\in
L^2(\Omega)$, $\g\in\LL^2(\Omega)$.

For our results we make use of the following assumptions:
We suppose that the coefficients $\CC$, $\bbeta$, $\gamma$ and the domain $\Omega$ are such that for $f,g\in
L^2(\Omega)$, $\ff,\g\in \HH^1(\TT)\cap\Hdivset\Omega$
the unique solutions $(u,\ssigma)\in H_0^1(\Omega)\times \Hdivset\Omega$ resp. $(v,\ttau)\in H_0^1(\Omega)\times \Hdivset\Omega$
of~\eqref{eq:model} resp.~\eqref{eq:adjoint} satisfy
\begin{subequations}\label{eq:assReg}
\begin{align}
  \norm{u}{H^2(\Omega)} + \norm{\ssigma}{\HH^1(\TT)} &\leq C (\norm{f}{} + \norm{\ff}{\HH^1(\TT)}), \\
  \norm{v}{H^2(\Omega)} + \norm{\ttau}{\HH^1(\TT)} & \leq C (\norm{g}{} + \norm{\g}{\HH^1(\TT)}).\label{eq:assReg:b}
\end{align}
\end{subequations}
Here, $\norm{\cdot}{H^s(\Omega)}$ is the usual notation for norms in the Sobolev space $H^s(\Omega)$ ($s>0$), and
$\norm{\cdot}{}$ is the $L^2(\Omega)$ norm and $\norm\cdot{\HH^s(\TT)}$ the broken Sobolev norm for vector valued
functions.

\begin{remark}
  The regularity estimates~\eqref{eq:assReg} are satisfied if $d=2$, $\CC$ is the identity matrix, $\bbeta\in
  C^1(\TT)^d\cap\Hdivset\Omega$ and $\Omega$ is convex. This can be seen as follows: The first component $u\in
  H_0^1(\Omega)$ of the solution of~\eqref{eq:model} satisfies
  \begin{align*}
    -\Delta u = f -\div\ff -(\div\bbeta)u + \bbeta\cdot\nabla u - \gamma u \in L^2(\Omega).
  \end{align*}
  Then $u\in H^2(\Omega)$ and $\norm{u}{H^2(\Omega)}$ is bounded by the $L^2(\Omega)$ norm of the right-hand side,
  since $\Omega$ is a convex polyhedral domain, see~\cite{grisvard}.
  Finally, the second equation of the model problem~\eqref{eq:model} shows 
  \begin{align*}
    \norm{\ssigma}{\HH^1(\TT)} = \norm{\ff - \nabla u + \bbeta u}{\HH^1(\TT)} 
    \lesssim \norm{\ff}{\HH^1(\TT)} + \norm{u}{H^2(\Omega)} \lesssim \norm{f}{} + \norm{\ff}{\HH^1(\TT)}.
  \end{align*}
  Similarly, one shows~\eqref{eq:assReg:b} (even a less regular coefficient $\bbeta$ suffices for the adjoint problem).
\end{remark}

\subsection{Assumptions on coefficients and test norms}\label{sec:ass}
Besides the assumptions on the coefficients and the domain to ensure unique solvability of the
problems~\eqref{eq:model},~\eqref{eq:adjoint} and the estimates~\eqref{eq:assReg} we also need some
additional assumptions on the coefficients that are listed in the following table:
\begin{table}[H]
\centering
\begin{tabular}{|N|c|c|c|c|}
\hline
\multicolumn{1}{|c|}{Case} & Test norm $\norm{\cdot}V$ & $\CC$ & $\bbeta$ & $\gamma$  \\ \hline\hline
\label{case:a} & $\VnormQopt\cdot$ & $C^1(\TT)^{d\times d}$ & $C^1(\TT)^d$ & $C^1(\TT)$ \\ \hline
\label{case:b} & $\VnormStd\cdot$ & $C^1(\TT)^{d\times d}$ & $0$ & $C^1(\TT)$\\ \hline
\label{case:c} & $\VnormSimple\cdot$ & $C^{0,1}(\overline\Omega)^{d\times d} \cap C^1(\TT)^{d\times d}$ & $0$ & $C^1(\TT)$ \\ \hline
\end{tabular}
\caption{Additional assumptions (besides~\eqref{eq:ass:coeff} and~\eqref{eq:assReg})
on the coefficients for the three test norms under consideration.} \label{tab:ass}
\end{table}
We emphasize that $\bbeta=0$ in the Cases~\ref{case:b},\ref{case:c} is also necessary in general. 
In particular, in Section~\ref{sec:ex} we provide a simple example where $\bbeta\neq 0$ and the choice $\norm{\vv}V =
\VnormStd\vv$ or $\norm\vv{V} = \VnormSimple\vv$
does not lead to higher convergence rates, whereas $\norm\vv{V}=\VnormQopt\vv$ does.

\subsection{$L^2(\Omega)$ projection}
Our first main result shows that the DPG method with ultra-weak formulation delivers up to a higher order term the
$L^2(\Omega)$ best approximation for the scalar field variable.
To that end let $\Pi^p : L^2(\Omega)\to \PP^p(\TT)$ denote the $L^2(\Omega)$ projector.
\begin{theorem}\label{thm:L2}
  Consider one of the Cases~\ref{case:a},~\ref{case:b}, or~\ref{case:c}.
  Let $\uu=(u,\ssigma,\wat u,\wat\sigma)\in U$ be the solution of~\eqref{eq:uwf} for some given $f\in L^2(\Omega)$,
  $\ff\in\LL^2(\Omega)$ and suppose $u\in H^{p+2}(\Omega)$, $\ssigma\in \HH^{p+1}(\TT)$.
  Let $\uu_h=(u_h,\ssigma_h,\wat u_h,\wat\sigma_h)\in U_h:=U_{hp}$ be the solution of the practical DPG
  method~\eqref{eq:practicalDPG}.
  Suppose $\PP_{c,0}^1(\TT)\times \RT^{p}(\TT) \subseteq V_{hk}$.
  It holds that
  \begin{align*}
    \norm{u-\Pi^p u}{} \leq \norm{u-u_h}{} \leq \norm{u-\Pi^p u}{} + C h^{p+2} (\norm{u}{H^{p+2}(\Omega)} +
    \norm{\ssigma}{\HH^{p+1}(\TT)}).
  \end{align*}
  The constant $C>0$ depends only on $\Omega$, $\CC$, $\bbeta$, $\gamma$, $p\in\N_0$, and shape-regularity of $\TT$.
\end{theorem}

\subsection{Higher convergence rate by increasing polynomial degree}
Our second main result shows that higher convergence rates for the scalar field variable are obtained by increasing the
polynomial degree in the approximation space.
\begin{theorem}\label{thm:augment}
  Consider one of the Cases~\ref{case:a},~\ref{case:b}, or~\ref{case:c}.
  Let $\uu=(u,\ssigma,\wat u,\wat\sigma)\in U$ be the solution of~\eqref{eq:uwf} for some given $f\in L^2(\Omega)$,
  $\ff\in\LL^2(\Omega)$ and suppose $u\in H^{p+2}(\Omega)$, $\ssigma\in \HH^{p+1}(\TT)$.
  Let $\uu_h^+=(u_h^+,\ssigma_h,\wat u_h,\wat\sigma_h)\in U_h:=U_{hp}^+$ 
  be the solution of the practical DPG method~\eqref{eq:practicalDPG}.
  Suppose $\PP_{c,0}^1(\TT)\times \RT^{p+1}(\TT) \subseteq V_{hk}$.
  It holds that
  \begin{align*}
    \norm{u-u_h^+}{} \leq C h^{p+2} (\norm{u}{H^{p+2}(\Omega)} + \norm{\ssigma}{\HH^{p+1}(\TT)}).
  \end{align*}
  The constant $C>0$ depends only on $\Omega$, $\CC$, $\bbeta$, $\gamma$, $p\in\N_0$, and shape-regularity of $\TT$.
\end{theorem}

\subsection{Higher convergence rate by postprocessing}\label{sec:main:postproc}
Our third and final main result shows that higher convergence rates for the scalar field variable are obtained by
postprocessing the solution:
Let $\uu_h=(u_h,\ssigma_h,\wat u_h,\wat\sigma_h)\in U_h:=U_{hp}$ be the solution of~\eqref{eq:practicalDPG}.
We define $\widetilde u_h \in \PP^{p+1}(\TT)$ on each element $T\in\TT$ as the solution of the local Neumann problem
\begin{subequations}\label{eq:postproc}
\begin{align}
  \ip{\nabla \widetilde u_h}{\nabla v_h}_T &= \ip{\CC\ff-\CC\ssigma_h+\bbeta u_h}{\nabla v_h}_T \quad\text{for all }
  v_h\in\PP^{p+1}(T), \\
  \ip{\widetilde u_h}1_T &= \ip{u_h}1_T.\label{eq:postproc:b}
\end{align}
\end{subequations}
Let us note that this type of postprocessing is common in literature and can already be found in the early 
work~\cite{stenbergPostPr}.

\begin{theorem}\label{thm:postproc}
  Consider one of the Cases~\ref{case:a},~\ref{case:b}, or~\ref{case:c}.
  Let $\uu=(u,\ssigma,\wat u,\wat\sigma)\in U$ be the solution of~\eqref{eq:uwf} for some given $f\in L^2(\Omega)$,
  $\ff\in\LL^2(\Omega)$ and suppose $u\in H^{p+2}(\Omega)$, $\ssigma\in \HH^{p+1}(\TT)$.
  Let $\uu_h=(u_h,\ssigma_h,\wat u_h,\wat\sigma_h)\in U_h:=U_{hp}$ 
  be the solution of the practical DPG method~\eqref{eq:practicalDPG} and define $\widetilde u_h \in\PP^{p+1}(\TT)$
  by~\eqref{eq:postproc}.
  Suppose $\PP_{c,0}^1(\TT)\times \RT^{p}(\TT) \subseteq V_{hk}$.
  It holds that
  \begin{align*}
    \norm{u-\widetilde u_h}{} \leq C h^{p+2} (\norm{u}{H^{p+2}(\Omega)} + \norm{\ssigma}{\HH^{p+1}(\TT)}).
  \end{align*}
  The constant $C>0$ depends only on $\Omega$, $\CC$, $\bbeta$, $\gamma$, $p\in\N_0$, and shape-regularity of $\TT$.
\end{theorem}

\section{Proofs}\label{sec:proof}
In this section we proof the results stated in Theorems~\ref{thm:L2},~\ref{thm:augment}, and~\ref{thm:postproc}.
First, in Section~\ref{sec:approx} we collect some standard results on projection operators and consider approximation
results with respect to $\norm\cdot{U}$.
Second, Section~\ref{sec:mixed} recalls the equivalent mixed formulation of the practical DPG method.
Then, Section~\ref{sec:aux} provides auxiliary results that allow to prove the main results in a uniform fashion.
Finally, in Sections~\ref{proof:L2},~\ref{proof:augment},~\ref{proof:postproc} we give the proofs of our main results.

\subsection{Projection operators and approximation results}\label{sec:approx}
Throughout let $p\in\N_0$.
Let $\Pi^p: L^2(\Omega) \to \PP^p(\TT)$ denote the $L^2(\Omega)$ projector. For $\ttau\in \LL^2(\Omega)$ the term $\Pi^p
\ttau$ is understood as the application of $\Pi^p$ to each component.
We have the (local) approximation properties
\begin{subequations}\label{eq:approx}
\begin{align}\label{eq:approx:L2}
  \norm{u-\Pi^p u}{} \leq C_p h^{p+1}\snorm{u}{H^{p+1}(\TT)} \quad\text{and}\quad
  \norm{\ssigma-\Pi^p\ssigma}{} \leq 
  C_ph^{p+1}\snorm{\ssigma}{\HH^{p+1}(\TT)},
\end{align}
where $\snorm{\cdot}{H^{n}(\TT)} := \norm{D_\TT^n\cdot}{}$ with $D_\TT^n$ denoting the $\TT$-elementwise $n$-th derivative
operator.
Let $\Pi_\nabla^{p+1} : H_0^1(\Omega)\to \PP_{c,0}^{p+1}(\TT)$ denote the Scott-Zhang projection operator or any other
operator with the property
\begin{align}\label{eq:approx:SZ}
  \norm{u-\Pi_\nabla^{p+1}u}{H^1(\Omega)} \leq C_p h^{p+1}\norm{u}{H^{p+2}(\Omega)}.
\end{align}
Moreover, let $\Pi_\div^p : \Hdivset\Omega\cap\HH^1(\TT) \to \RT^p(\TT)$ denote the Raviart-Thomas operator, which
satisfies
\begin{align}\label{eq:approx:RT}
  \norm{\ssigma-\Pi_\div^p\ssigma}{} \leq C_p h^{k+1}\snorm{\ssigma}{\HH^{k+1}(\TT)} \quad\text{for }k\in [0,p],
\end{align}
and the commutativity property $\div\Pi_\div^p\ssigma = \Pi^p\div\ssigma$.
Note that $\Pi_\div^p$ is well-defined for functions $\ssigma\in \Hdivset\Omega\cap\HH^1(\TT)$: First,
normal traces of $\ssigma\in\HH^1(\TT)$ are well-defined on each facet of $\partial T$, $T\in\TT$, in the sense of
$L^2(\partial T)$, i.e.,
$\ssigma\cdot\normal_T \in L^2(\partial T)$ and, second, $\ssigma\in\Hdivset\Omega$ implies unisolvency of normal
traces.
The constant $C_p>0$ in~\eqref{eq:approx} depends only on $p\in\N_0$ and shape-regularity of $\TT$.
\end{subequations}

The following result is an adaptation of~\cite[Theorem~5 and Corollary~6]{SupconvDPG}.
\begin{theorem}\label{thm:approx}
  Let $p\in\N_0$ and let $w\in H^{p+2}(\Omega)$, $\cchi\in \HH^{p+1}(\TT)\cap\Hdivset\Omega$. Define $\ww :=
  (w,\cchi,\gamma_{0,\cS}w,\gamma_{\normal,\cS}\cchi) \in U$.
  If $U_h\in\{U_{hp},U_{hp}^+\}$, then
  \begin{align*}
    \min_{\ww_h\in U_h}\norm{\ww-\ww_h}U \leq C h^{p+1}(\norm{w}{H^{p+2}(\Omega)} + \norm{\cchi}{\HH^{p+1}(\TT)}).
  \end{align*}
  The constant $C>0$ depends only on $p$ and shape-regularity of $\TT$.
\end{theorem}
\begin{proof}
  Define
  \begin{align*}
    \ww_h := (\Pi^p w,\Pi^p\cchi,\gamma_{0,\cS}\Pi_\nabla^{p+1}w,\gamma_{\normal,\cS}\Pi_\div^p\cchi)\in U_h.
  \end{align*}
  We estimate the terms in
  \begin{align*}
    \norm{\ww-\ww_h}U^2 = \norm{w-\Pi^pw}{}^2 + \norm{\cchi-\Pi^p\cchi}{}^2 +
    \norm{\gamma_{0,\cS}(w-\Pi_\nabla^{p+1}w)}{1/2,\cS}^2 + 
    \norm{\gamma_{\normal,\cS}(\cchi-\Pi_\div^p\cchi)}{-1/2,\cS}^2.
  \end{align*}
  First, following the lines of~\cite[Proof of Theorem~5]{SupconvDPG} shows
  \begin{align*}
    \norm{\gamma_{\normal,\cS}(\cchi-\Pi_\div^p\cchi)}{-1/2,\cS} \lesssim h^{p+1}\norm{\cchi}{\HH^{p+1}(\TT)}.
  \end{align*}
  This can be done since all the essential arguments in~\cite[Proof of Theorem~5]{SupconvDPG} are local.
  Then, observe that $\norm{\gamma_{0,\cS}(\cdot)}{1/2,\cS}\leq \norm{\cdot}{H^1(\Omega)}$ by definition of the norms.
  Finally, applying the approximation properties~\eqref{eq:approx:L2}--\eqref{eq:approx:SZ} and putting altogether
  finishes the proof.
\end{proof}

\subsection{Mixed formulation of practical DPG method}\label{sec:mixed}
The practical DPG method~\eqref{eq:practicalDPG} can be reformulated as a mixed problem, see,
e.g.~\cite{BoumaGH_DPGconvRates}. Recall that we made the assumption of the existence of a Fortin
operator~\eqref{eq:fortin}. 
The mixed DPG formulation then reads:
Find $(\uu_h,\eeps_{hk})\in U_h\times V_{hk}$ such that
\begin{subequations}\label{eq:mixedDPG}
  \begin{alignat}{4}
    &\ip{\eeps_{hk}}{\vv_{hk}}_V + b(\uu_h,\vv_{hk}) &\,=\,& F(\vv_{hk}) &\qquad&\text{for all } \vv_{hk}\in V_{hk},
    \label{eq:mixedDPG:a}\\
    &b(\ww_h,\eeps_{hk}) &\,=\,& 0 &\qquad&\text{for all } \ww_h \in U_h. \label{eq:mixedDPG:b}
  \end{alignat}
\end{subequations}
The function $\eeps_{hk}\in V_{hk}$ is called the error function and it holds
\begin{align*}
  \norm{\eeps_{hk}}V \lesssim \norm{\uu-\uu_h}U,
\end{align*}
under the assumption~\eqref{eq:fortin}, see~\cite[Theorem~2.1]{DPGaposteriori}. 
Note that the solution $\uu_h$ in~\eqref{eq:mixedDPG} is identical to the solution of~\eqref{eq:practicalDPG}.
Setting $\eeps:=0$ we have that $(\uu,\eeps)\in U\times V$ satisfies the mixed formulation for all test functions
$(\ww,\vv)\in U\times V$.
In particular, we have Galerkin orthogonality
\begin{align}\label{eq:orthogonality}
  a( (\uu-\uu_h),(\eeps-\eeps_{hk}), (\ww_h,\vv_{hk}) ) = 0 \quad\text{for all } (\ww_h,\vv_{hk}) \in U_h\times V_{hk},
\end{align}
where $a( (\ww,\vv), (\delta\ww,\delta\vv) ) := b(\ww,\delta\vv) + \ip{\vv}{\delta\vv}_V - b(\delta\ww,\vv)$ for all
$\ww,\delta\ww\in U$, $\vv,\delta\vv\in V$.

\subsection{Auxiliary results}\label{sec:aux}
Recall the adjoint problem~\eqref{eq:adjoint} with $g\in L^2(\Omega)$, $\g=0$,
\begin{align}\label{eq:adjoint:2}
\begin{split}
  -\div\ttau - \bbeta\cdot\ttau + \gamma v &= g, \\
  \nabla v - \CC\ttau &= 0, \\
  v|_\Gamma &=0.
\end{split}
\end{align}
Note that $\vv = (v,\ttau)\in H_0^1(\Omega)\times \Hdivset\Omega\subset V$.
In particular, there exists a unique $\ww\in U$ with $\Theta\ww = \vv$, since $\Theta : U \to V$ is an isomorphism.
Note that by the definition of the trial-to-test operator~\eqref{def:ttt}, the element $\ww$ depends on the choice of scalar
products in $V$.
This is investigated in the following result.
\begin{lemma}\label{lem:w:rep}
  Let $g\in L^2(\Omega)$ and let $\vv:=(v,\ttau)\in H_0^1(\Omega)\times \Hdivset\Omega$ denote the solution
  of~\eqref{eq:adjoint:2}.
  The unique element $\ww\in U$ with $\Theta\ww = \vv$ has the following representation depending on the cases
  from Section~\ref{sec:ass}:
  \begin{itemize}
    \item Case~\ref{case:a} ($\norm\cdot{V} = \VnormQopt\cdot$)
      \begin{align*}
        \ww = (g,0,0,0) + (u^*,\ssigma^*,\gamma_{0,\cS}u^*,\gamma_{\normal,\cS}\ssigma^*),
      \end{align*}
      where $(u^*,\ssigma^*)\in H_0^1(\Omega)\times \Hdivset\Omega$ solves~\eqref{eq:model} with $f = v$ and $\ff =
      \ttau$.
    \item Case~\ref{case:b} ($\norm\cdot{V} = \VnormStd\cdot$)
      \begin{align*}
        \ww = (g-\gamma v,0,0,\gamma_{\normal,\cS}\ttau) 
        + (u^*,\ssigma^*,\gamma_{0,\cS}u^*,\gamma_{\normal,\cS}\ssigma^*),
      \end{align*}
      where $(u^*,\ssigma^*)\in H_0^1(\Omega)\times \Hdivset\Omega$ solves~\eqref{eq:model} with $f = \gamma(\gamma
      v-g)-\div\ttau+v$ and $\ff = \ttau$.
    \item Case~\ref{case:c} ($\norm\cdot{V} = \VnormSimple\cdot$)
      \begin{align*}
        \ww = (g-\gamma v,0,0,\gamma_{\normal,\cS}(\CC\ttau)) 
        + (u^*,\ssigma^*,\gamma_{0,\cS}u^*,\gamma_{\normal,\cS}\ssigma^*),
      \end{align*}
      where $(u^*,\ssigma^*)\in H_0^1(\Omega)\times \Hdivset\Omega$ solves~\eqref{eq:model} with $f = \gamma(\gamma
      v-g)-\div(\CC\ttau)+v$ and $\ff = \CC^{-1}\ttau$.
  \end{itemize}
  Moreover,
  \begin{align}\label{eq:reg:dual}
    \norm{v}{H^2(\Omega)} + \norm{\ttau}{\HH^1(\TT)} + \norm{u^*}{H^2(\Omega)} + \norm{\ssigma^*}{\HH^1(\TT)}
    \leq C \norm{g}{}.
  \end{align}
  For Case~\ref{case:c} it also holds that $\ttau,\ssigma^*\in \HH^1(\Omega)$.
\end{lemma}

\begin{proof}
  \textbf{Case~\ref{case:a} }
  Recall that $\ip{\Theta\ww}{(\mu,\llambda)}_V = b(\ww,(\mu,\llambda))$ for all $(\mu,\llambda)\in V$.
  With the inner product in $V$ and $\pwdiv\ttau=\div\ttau$, $\pwnabla v = \nabla v$ we have for $(\mu,\llambda)\in V$
  that
  \begin{align*}
    \ip\vv{(\mu,\llambda)}_V &= 
    \ip{-\div\ttau-\bbeta\cdot\ttau+\gamma v}{-\pwdiv\llambda-\bbeta\cdot\llambda+\gamma\mu}
    \\ &\qquad + \ip{\CC^{1/2}\ttau-\CC^{-1/2}\nabla v}{\CC^{1/2}\llambda-\CC^{-1/2}\pwnabla\mu} \\
    &\qquad + \ip{\CC\ttau}\llambda + \ip{v}\mu.\\
    &= \ip{g}{-\pwdiv\llambda-\bbeta\cdot\llambda+\gamma\mu} + \ip{\CC\ttau}\llambda + \ip{v}\mu \\
    &= b((g,0,0,0),(\mu,\llambda)) +  \ip{\CC\ttau}\llambda + \ip{v}\mu.
  \end{align*}
  Let $(u^*,\ssigma^*)\in H_0^1(\Omega)\times \Hdivset\Omega$ solve the (primal) problem~\eqref{eq:model} with $f = v\in
  L^2(\Omega)$ and $\ff = \ttau \in\Hdivset\Omega$. In particular, $(u^*,\ssigma^*)$ solves the ultra-weak formulation~\eqref{eq:uwf},
  i.e.,
  \begin{align*}
    b((u^*,\ssigma^*,\gamma_{0,\cS}u^*,\gamma_{\normal,\cS}\ssigma^*),(\mu,\llambda)) = \ip{\CC\ttau}\llambda + \ip{v}\mu
    \quad\text{for all }(\mu,\llambda)\in V.
  \end{align*}
  Defining $\ww := (g,0,0,0) + (u^*,\ssigma^*,\gamma_{0,\cS}u^*,\gamma_{\normal,\cS}\ssigma^*)$ and putting altogether
  shows
  \begin{align*}
    \ip\vv{(\mu,\llambda)}_V = b( (g,0,0,0),(\mu,\llambda)) + 
    b( (u^*,\ssigma^*,\gamma_{0,\cS}u^*,\gamma_{\normal,\cS}\ssigma^*),(\mu,\llambda)) = b(\ww,(\mu,\llambda)).
  \end{align*}
  Thus, $\Theta\ww=\vv$.

  \noindent
  \textbf{Case~\ref{case:b} }
  The scalar product in this case is given by
  \begin{align*}
    \ip{(v,\ttau)}{(\mu,\llambda)}_V = \ip{\div \ttau}{\pwdiv\llambda} + \ip{\CC\ttau}{\llambda}
    + \ip{\CC^{-1}\nabla v}{\pwnabla\mu} + \ip{v}\mu.
  \end{align*}
  Recall that $\bbeta=0$ and note that $\div\ttau = -g +\gamma v$ by~\eqref{eq:adjoint:2}. 
  Therefore,
  \begin{align*}
    \ip{\div\ttau}{\pwdiv\llambda} = \ip{g-\gamma v}{-\pwdiv\llambda} 
    &= \ip{g-\gamma v}{-\pwdiv\llambda+\gamma \mu} + \ip{\gamma(\gamma v-g)}{\mu} \\
    &= b( (g-\gamma v,0,0,0),(\mu,\llambda)) + \ip{\gamma(\gamma v-g)}{\mu}.
  \end{align*}
  With $\CC\ttau=\nabla v$ and piecewise integration by parts we obtain
  \begin{align*}
    \ip{\CC^{-1}\nabla v}{\pwnabla \mu} = \ip{\ttau}{\pwnabla \mu} &= 
    \dual{\gamma_{\normal,\cS}\ttau}\mu_\cS + \ip{-\div\ttau}{\mu} \\
    &=b((0,0,0,\gamma_{\normal,\cS}\ttau),(\mu,\llambda)) + \ip{-\div\ttau}{\mu}.
  \end{align*}
  Thus, 
  \begin{align*}
    \ip{(v,\ttau)}{(\mu,\llambda)}_V &= \ip{\div \ttau}{\pwdiv\llambda} + \ip{\CC\ttau}{\llambda}
    + \ip{\CC^{-1}\nabla v}{\pwnabla\mu} + \ip{v}\mu \\
    &= b( (g-\gamma v,0,0,\gamma_{\normal,\cS}\ttau),(\mu,\llambda)) 
    + \ip{\gamma(\gamma v-g)-\div\ttau+v}\mu + \ip{\CC\ttau}\llambda.
  \end{align*}
  Defining $\ww := (g-\gamma v,0,0,\gamma_{\normal,\cS}\ttau) 
  + (u^*,\ssigma^*,\gamma_{0,\cS}u^*,\gamma_{\normal,\cS}\ssigma^*)$, where $(u^*,\ssigma^*)$
  solves~\eqref{eq:model} with data $f = \gamma(\gamma v-g)-\div\ttau+v$, $\ff = \ttau$, shows
  \begin{align*}
    \ip{(v,\ttau)}{(\mu,\llambda)}_V  = b(\ww,(\mu,\llambda)) \quad\text{for all }(\mu,\llambda)\in V.
  \end{align*}

  \noindent
  \textbf{Case~\ref{case:c} }
  The proof is similar as for Case~\ref{case:b}. Thus, we only give details on the important differences.
  We have to take care of the terms involving the matrix $\CC$. 
  Note that by the assumptions on $\CC$ it holds $\CC^{-1}\ttau
  \in \Hdivset\Omega$ and $\CC\ttau\in\Hdivset\Omega$ as well. We have
  \begin{align*}
    \ip{\ttau}{\llambda} = \ip{\CC\CC^{-1}\ttau}\llambda,
  \end{align*}
  and using $\CC\ttau=\nabla v$ and integration by parts,
  \begin{align*}
    \ip{\nabla v}{\pwnabla \mu} = \ip{\CC\ttau}{\pwnabla\mu} = \dual{\gamma_{\normal,\cS}(\CC\ttau)}\mu_\cS 
    - \ip{\div(\CC\ttau)}\mu.
  \end{align*}
  Defining $\ww := (g-\gamma v,0,0,\gamma_{\normal,\cS}(\CC\ttau)) 
  + (u^*,\ssigma^*,\gamma_{0,\cS}u^*,\gamma_{\normal,\cS}\ssigma^*)$, where $(u^*,\ssigma^*)$
  solves~\eqref{eq:model} with data $f = \gamma(\gamma v-g)-\div(\CC\ttau)+v$, $\ff = \CC^{-1}\ttau$, shows
  \begin{align*}
    \ip{(v,\ttau)}{(\mu,\llambda)}_V  &= \ip{\div \ttau}{\pwdiv\llambda} + \ip{\ttau}{\llambda}
    + \ip{\nabla v}{\pwnabla\mu} + \ip{v}\mu \\
    &= b(\ww,(\mu,\llambda)) \quad\text{for all }(\mu,\llambda)\in V.
  \end{align*}

  Finally, note that for all three cases it is straightforward to prove $\norm{f}{}+\norm{\ff}{\HH^1(\TT)} \lesssim \norm{g}{}$.
  Then, \eqref{eq:assReg} shows~\eqref{eq:reg:dual}. Moreover, in Case~\ref{case:c} we have $\ttau = \CC^{-1}\nabla v \in
  \HH^1(\Omega)$ and $\ff = \CC^{-1}\ttau\in\HH^1(\Omega)$, thus, $\ssigma^* = \CC^{-1}\ttau-\CC^{-1}\nabla u^*\in\HH^1(\Omega)$.
  This finishes the proof.
\end{proof}

\begin{lemma}\label{lem:est}
  Let $\uu = (u,\ssigma,\wat u,\wat\sigma) \in U$ denote the solution of~\eqref{eq:uwf} and let
  $\uu_h=(u_h,\ssigma_h,\wat u_h,\wat\sigma_h) \in U_h
  \in\{U_{hp},U_{hp}^+\}$ denote the solution of~\eqref{eq:practicalDPG}.
  Suppose $(g,0,0,0)\in U_h$, i.e., $g\in\PP^p(\TT)$ if $U_{h}=U_{hp}$ resp. $g\in\PP^{p+1}(\TT)$ if
  $U_h=U_{hp}^+$.
  Moreover, suppose that
  \begin{itemize}
    \item $\PP_{c,0}^1(\TT)\times\RT^{p}(\TT)\subset V_{hk}$ if $U_h=U_{hp}$,
    \item $\PP_{c,0}^1(\TT)\times\RT^{p+1}(\TT)\subset V_{hk}$ if $U_h=U_{hp}^+$.
  \end{itemize}
  It holds that
  \begin{align*}
    \ip{u-u_h}{g} \leq C h \norm{\uu-\uu_h}U \norm{g}{}.
  \end{align*}
  The constant $C>0$ only depends on $\Omega$, $\CC$, $\bbeta$, $\gamma$, $p\in\N_0$, and shape-regularity of $\TT$.
\end{lemma}
\begin{proof}
  Let $\vv=(v,\ttau)\in V$ denote the solution of the adjoint problem~\eqref{eq:adjoint:2} with the given $g\in L^2(\Omega)$.
  Let $\ww=\Theta^{-1}\vv\in U$ denote the element from Lemma~\ref{lem:w:rep}.
  Since $(v,\ttau)\in H_0^1(\Omega)\times\Hdivset\Omega$, the identities~\eqref{eq:jumps} and the adjoint
  problem~\eqref{eq:adjoint:2} imply that $\ip{u-u_h}{g} = b(\uu-\uu_h,\vv)$.
  With the bilinear form $a(\cdot,\cdot)$ of the mixed formulation of DPG (Section~\ref{sec:mixed}) and the fact that
  $b(\ww,\delta\vv) = \ip{\vv}{\delta\vv}_V = \ip{\delta\vv}\vv_V$ for all $\delta\vv\in V$, we infer
  \begin{align*}
    \ip{u-u_h}{g} = b(\uu-\uu_h,\vv) = a( (\uu-\uu_h,\eeps-\eeps_h), (\ww,\vv)).
  \end{align*}
  Here, $\eeps=0$ and $\eeps_h\in V_{hk}$ is the error function which satisfies $\norm{\eeps_h}V \lesssim
  \norm{\uu-\uu_h}U$ (see Section~\ref{sec:mixed}).
  This, Galerkin orthogonality and boundedness of the bilinear form $a(\cdot,\cdot)$ 
  show for arbitrary $(\ww_h,\vv_h)\in (U_h,V_{hk})$ that
  \begin{align*}
    \ip{u-u_h}{g} =a( (\uu-\uu_h,\eeps-\eeps_h), (\ww,\vv)) &= a( (\uu-\uu_h,\eeps-\eeps_h), (\ww-\ww_h,\vv-\vv_h))\\
    &\lesssim \norm{\uu-\uu_h}U\left(\norm{\ww-\ww_h}U + \norm{\vv-\vv_h}V\right).
  \end{align*}
  It remains to prove $\norm{\ww-\ww_h}U + \norm{\vv-\vv_h}V\lesssim h\norm{g}{}$.
  We estimate $\norm{\vv-\vv_h}{V}$ for all three cases simultaneously and handle the estimation of
  $\norm{\ww-\ww_h}U$ for the three cases separately, since the representation of $\ww$ by Lemma~\ref{lem:w:rep} depends
  on the choice of norms in $V$.

  We start with the estimation of $\norm{\vv-\vv_h}V$:
  We first consider $U_h=U_{hp}$. Note that $\PP_{c,0}^1(\TT)\times\RT^{p}(\TT)\subset V_{hk}$. Choose $\vv_h =
  (\Pi_\nabla^1 v,\Pi_\div^p\ttau)\in V_{hk}$. Recall that all norms under consideration are equivalent, i.e.,
  $\VnormQopt\cdot\simeq \VnormStd\cdot \simeq \VnormSimple\cdot$. Then, using the approximation
  properties~\eqref{eq:approx} together with~\eqref{eq:reg:dual}, we get
  \begin{align*}
    \norm{\vv-\vv_h}V &\simeq\VnormSimple{\vv-\vv_h} \leq \norm{v-\Pi_\nabla^1 v}{H^1(\Omega)} 
    + \norm{\ttau-\Pi_\div^p\ttau}{\Hdivset\Omega}
    \lesssim h\norm{g}{} + \norm{\div(\ttau-\Pi_\div^p\ttau)}{}.
  \end{align*}
  Then, for the remaining term the commutativity property of the Raviart-Thomas projection, the adjoint
  problem~\eqref{eq:adjoint:2} and $g\in\PP^p(\TT)$ yield
  \begin{align*}
    \norm{\div(\ttau-\Pi_\div^p\ttau)}{} = \norm{(1-\Pi^p)\div\ttau}{} = 
    \norm{(1-\Pi^p)(-g-\bbeta\cdot\ttau+\gamma v)}{} = 
    \norm{(1-\Pi^p)(\gamma v - \bbeta\cdot\ttau)}{}.
  \end{align*}
  Using the approximation properties of $\Pi^0$,  $\gamma\in C^1(\TT)$, $\bbeta\in C^1(\TT)^d$, and~\eqref{eq:reg:dual}
  shows
  \begin{align*}
    \norm{(1-\Pi^p)(\gamma v - \bbeta\cdot\ttau)}{} \leq \norm{(1-\Pi^0)(\gamma v - \bbeta\cdot\ttau)}{} 
    \lesssim h\norm{\pwnabla(\gamma v-\bbeta\cdot\ttau)}{} \lesssim h\norm{g}{}.
  \end{align*}
  Therefore, we obtain $\norm{\vv-\vv_h}V\lesssim h\norm{g}{}$. If $U_h=U_{hp}^+$, then we choose
  $\vv_h=(\Pi_\nabla^1,\Pi_\div^{p+1}\ttau)\in V_{hk}$. With the same lines of proof we also infer $\norm{\vv-\vv_h}V\lesssim
  h\norm{g}{}$.

  It only remains to estimate $\norm{\ww-\ww_h}U$.
  We distinguish between the three different cases:

  \noindent
  \textbf{Case~\ref{case:a} }
  By Lemma~\ref{lem:w:rep} we have $\ww = (g,0,0,0) + \widetilde\ww$, where 
  $\widetilde\ww = (u^*,\ssigma^*,\gamma_{0,\cS}u^*,\gamma_{\normal,\cS}\ssigma^*)$.
  We choose $\ww_h = (g,0,0,0) + \widetilde \ww_h$, where $\widetilde\ww_h\in U_{h0}\subseteq U_h$ 
  is the best-approximation of $(u^*,\ssigma^*,\gamma_{0,\cS}u^*,\gamma_{\normal,\cS}\ssigma^*)$ with respect to $\norm\cdot{U}$.
  From Proposition~\ref{prop:dpg} together with Theorem~\ref{thm:approx} and~\eqref{eq:reg:dual} it follows that
  \begin{align*}
    \norm{\ww-\ww_h}U = \norm{\widetilde\ww-\widetilde\ww_h}U \lesssim h\norm{g}{}.
  \end{align*}

  \noindent
  \textbf{Case~\ref{case:b} }
  By Lemma~\ref{lem:w:rep} we have $\ww = (g-\gamma v,0,0,\gamma_{\normal,\cS}\ttau) + \widetilde\ww$ and choose
  \begin{align*}
    \ww_h = (g-\Pi^0\gamma v,0,0,\gamma_{\normal,\cS}\Pi_\div^0\ttau) + \widetilde\ww_h,
  \end{align*}
  where $\widetilde\ww_h\in U_{h0}$ is the best approximation of $\widetilde\ww$ with respect to $\norm\cdot{U}$.
  Note that the same arguments as before lead to $\norm{\widetilde\ww-\widetilde\ww_h}U\lesssim h \norm{g}{}$.
  Therefore,
  \begin{align*}
    \norm{\ww-\ww_h}U &\leq \norm{(1-\Pi^0)\gamma v}{} + \norm{\gamma_{\normal,\cS}(\ttau-\Pi_\div^0\ttau)}{-1/2,\cS}
    + \norm{\widetilde\ww-\widetilde\ww_h}U \lesssim h\norm{g}{},
  \end{align*}
  where we used~\eqref{eq:approx} and the approximation property of $\gamma_{\normal,\cS}\Pi_\div^p$ 
  in the $H^{-1/2}(\cS)$ norm (see the proof of Theorem~\ref{thm:approx}) together with~\eqref{eq:reg:dual}.

  \noindent
  \textbf{Case~\ref{case:c} }
  The proof follows as for Case~\ref{case:b}. Therefore, we omit the details.
\end{proof}

\subsection{Proof of Theorem~\ref{thm:L2}}\label{proof:L2}
  The best approximation property of $\Pi^p$ and the triangle inequality show that
  \begin{align*}
    \norm{u-\Pi^pu}{}\leq\norm{u-u_h}{} \leq \norm{u-\Pi^p u}{} + \norm{\Pi^p(u-u_h)}{}.
  \end{align*}
  With $g:=\Pi^p u- u_h\in\PP^p(\TT)$ observe that
  \begin{align*}
    \norm{g}{}^2 = \ip{g}{g} = \ip{\Pi^p(u-u_h)}{g} = \ip{u-u_h}{g}.
  \end{align*}
  We apply Lemma~\ref{lem:est}, and the approximation result from Theorem~\ref{thm:approx} to see
  \begin{align*}
    \norm{g}{}^2 = \ip{u-u_h}{g} \lesssim h \norm{\uu-\uu_h}U \norm{g}{} 
    \lesssim h h^{p+1} (\norm{u}{H^{p+2}(\Omega)} + \norm{\ssigma}{\HH^{p+1}(\TT)})\norm{g}{}.
  \end{align*}
  Dividing by $\norm{g}{}$ we infer
  \begin{align*}
    \norm{u-u_h}{} \leq \norm{u-\Pi^p u}{} + \norm{g}{} \leq \norm{u-\Pi^p u}{} +
    Ch^{p+2}(\norm{u}{H^{p+2}(\Omega)} + \norm{\ssigma}{\HH^{p+1}(\TT)}),
  \end{align*}
  which finishes the proof.
\qed

\subsection{Proof of Theorem~\ref{thm:augment}}\label{proof:augment}
  The proof is similar to the one for Theorem~\ref{thm:L2}. We consider 
  \begin{align*}
    \norm{u-u_h^+}{} \leq \norm{u-\Pi^{p+1}u}{} + \norm{\Pi^{p+1}u-u_h^+}{}.
  \end{align*}
  Define $g:=\Pi^{p+1}u-u_h^+\in\PP^{p+1}(\TT)$. To estimate the second term we argue as in the proof of
  Theorem~\ref{thm:L2} to obtain $\norm{g}{} \lesssim h^{p+2}(\norm{u}{H^{p+2}(\Omega)} +
  \norm{\ssigma}{\HH^{p+1}(\TT)})$. 
  The first term is estimated with the approximation property~\eqref{eq:approx:L2} of the $L^2$ projection, i.e.,
  \begin{align*}
    \norm{u-\Pi^{p+1}u}{} \lesssim h^{p+2}\norm{u}{H^{p+2}(\Omega)}.
  \end{align*}
  This finishes the proof.
\qed

\subsection{Proof of Theorem~\ref{thm:postproc}}\label{proof:postproc}
Note that~\eqref{eq:postproc:b} is equivalent to $\Pi^0\widetilde u_h = \Pi^0 u_h$.
This yields
  \begin{align*}
    \norm{u-\widetilde u_h}{} \leq \norm{(1-\Pi^0)(u-\widetilde u_h)}{} 
    + \norm{\Pi^0(u-\widetilde u_h)}{} \lesssim h \norm{\pwnabla (u-\widetilde u_h)}{}+\norm{\Pi^0(u-u_h)}{},
  \end{align*}
  where we have used the local approximation property of $\Pi^0$.
  We define $g:=\Pi^0(u-u_h)$. Applying Lemma~\ref{lem:est} and Theorem~\ref{thm:approx} shows
  \begin{align*}
    \norm{g}{}^2 = \ip{\Pi^0(u-u_h)}{g} = \ip{u-u_h}{g} \lesssim h \norm{\uu-\uu_h}U\norm{g}{} \lesssim 
    h^{p+2}(\norm{u}{H^{p+2}(\Omega)}+\norm{\ssigma}{\HH^{p+1}(\TT)})\norm{g}{}.
  \end{align*}
  It remains to estimate $\norm{\pwnabla(u-\widetilde u_h)}{}$. The proof follows standard arguments from finite element
  analysis and is included for completeness.
  To that end define $\overline u_h\in\PP^{p+1}(\TT)$ as the solution of the auxiliary Neumann problem
  \begin{align*}
    \ip{\nabla \overline u_h}{\nabla v_h}_T &= \ip{\CC\ff-\CC\ssigma+\bbeta u}{\nabla v_h}_T \quad\text{for all }
    v_h\in\PP^{p+1}(T), \\
    \ip{\overline u_h}1_T &= 0
  \end{align*}
  for all $T\in\TT$.
  Then,
  \begin{align*}
    \norm{\pwnabla(\overline u_h-\widetilde u_h)}{}^2 &= \ip{-\CC(\ssigma-\ssigma_h)+\bbeta(u-u_h)}{\pwnabla(\overline
    u_h-\widetilde u_h)} \\
    &\lesssim \norm{\uu-\uu_h}U \norm{\pwnabla(\overline u_h-\widetilde u_h)}{}
    \lesssim h^{p+1}(\norm{u}{H^{p+2}(\Omega)}+\norm{\ssigma}{\HH^{p+1}(\TT)})
    \norm{\pwnabla(\overline u_h-\widetilde u_h)}{}.
  \end{align*}
  To estimate $\norm{\pwnabla(u-\overline u_h)}{}$ note that there holds Galerkin orthogonality
  $\ip{\pwnabla(u-\overline u_h)}{\pwnabla v_h} = 0$ for all $v_h\in\PP^{p+1}(\TT)$. Hence, 
  standard approximation results show
  \begin{align*}
    \norm{\pwnabla(u-\overline u_h)}{} = \min_{v_h\in\PP^{p+1}(\TT)} \norm{\pwnabla(u-v_h)}{}
    \lesssim h^{p+1}\norm{u}{H^{p+2}(\Omega)}.
  \end{align*}
  Putting altogether gives
  \begin{align*}
    \norm{u-\widetilde u_h}{} &\lesssim h\norm{\pwnabla(u-\widetilde u_h)}{} + \norm{g}{} \\
    &\lesssim 
    h(\norm{\pwnabla(u-\overline u_h)}{}+\norm{\pwnabla(\overline u_h-\widetilde u_h)}{})
    + h^{p+2}(\norm{u}{H^{p+2}(\Omega)}+\norm{\ssigma}{\HH^{p+1}(\TT)}) \\
    &\lesssim h^{p+2}(\norm{u}{H^{p+2}(\Omega)}+\norm{\ssigma}{\HH^{p+1}(\TT)}),
  \end{align*}
  which finishes the proof.
\qed


\section{Numerical Studies}\label{sec:ex}
In this section we present results of two numerical examples. 
Let $\Omega = (0,1)^2$ be a squared domain. Throughout we consider the manufactured solution
\begin{align*}
  u(x,y) = \sin(\pi x)\sin(\pi y), \quad(x,y)\in\Omega,
\end{align*}
which is smooth and satisfies $u|_\Gamma = 0$.

Let $\uu_h = (u_h,\ssigma_h,\wat u_h,\wat\sigma_h)\in U_{hp}$ and $\uu_h^+ =(u_h^+,\ssigma_h^+,\wat
u_h^+,\wat\sigma_h^+)\in U_{hp}^+$ denote the solutions of the practical DPG method~\eqref{eq:practicalDPG}
and let $\widetilde u_h\in\PP^{p+1}(\TT)$ be the postprocessed solution of $\uu_h$, see Section~\ref{sec:main:postproc}.
We present results for $p=0,1,2,3$, where we use the test space
\begin{align*}
  V_{hk} := \PP^{p+2}(\TT)\times \PP^{p+2}(\TT)^d.
\end{align*}
To verify our main results (Theorem~\ref{thm:L2}, Theorem~\ref{thm:augment}, and Theorem~\ref{thm:postproc}) we check
the convergence rates of the $L^2$ errors
\begin{align*}
  \norm{\Pi^p u - u_h}{}, \quad \norm{u-u_h^+}{}, \quad\text{and}\quad \norm{u-\widetilde u_h}{}.
\end{align*}
In all examples below we choose $\CC$ to be the identity matrix. Thus, $\VnormStd\cdot=\VnormSimple\cdot$ and
Cases~\ref{case:b},~\ref{case:c} are identical. The other coefficients are chosen such that 
the regularity assumptions~\ref{eq:assReg} are satisfied.

All computations start with the initial triangulation $\TT_1$ visualized in Figure~\ref{fig:mesh}.
\begin{figure}[htb]
  \begin{center}
    \begin{tikzpicture}
\begin{axis}[
    axis equal,
    width=0.45\textwidth,
    xlabel={$x$},
    ylabel={$y$},
]

\addplot[patch,color=white,
faceted color = black, line width = 1.5pt,
patch table ={elements.dat}] file{coordinates.dat};
\addplot[mark=*,color=gray,only marks] table[x index=0,y index=1] {coordinates.dat};
\end{axis}
\end{tikzpicture}
  \end{center}
  \caption{Initial triangulation $\TT_1$ of domain $\Omega=(0,1)^2$.}
  \label{fig:mesh}
\end{figure}
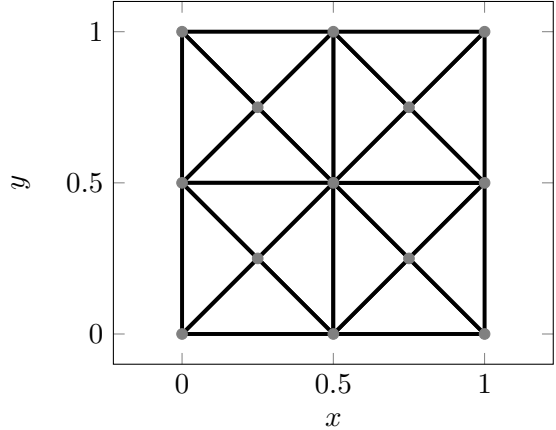

\subsection{Example 1}\label{sec:exp1}

\begin{table}[htb]
  \centering
  \begin{tabular}{|c|c|cc|cc|cc|cc|}
\hline
$p$ & $\#\TT$ & $\norm{u-u_h}{}$ & rate & $\norm{\Pi^p u-u_h}{}$ & rate & $\norm{u-u_h^+}{}$ & rate & $\norm{u-\widetilde u_h}{}$ & rate
\\ \hline\hline
\multirow{7}{*}{0}
& 16 & 1.94e-01 & --- & 7.41e-02 & --- & 8.37e-02 & --- & 1.23e-01 & --- \\
& 64 & 9.37e-02 & 1.05 & 1.85e-02 & 2.00 & 2.09e-02 & 2.01 & 3.21e-02 & 1.94 \\
& 256 & 4.64e-02 & 1.01 & 4.63e-03 & 2.00 & 5.20e-03 & 2.00 & 8.12e-03 & 1.98 \\
& 1024 & 2.32e-02 & 1.00 & 1.16e-03 & 2.00 & 1.30e-03 & 2.00 & 2.04e-03 & 2.00 \\
& 4096 & 1.16e-02 & 1.00 & 2.90e-04 & 2.00 & 3.25e-04 & 2.00 & 5.09e-04 & 2.00 \\
& 16384 & 5.79e-03 & 1.00 & 7.24e-05 & 2.00 & 8.13e-05 & 2.00 & 1.27e-04 & 2.00 \\
& 65536 & 2.89e-03 & 1.00 & 1.81e-05 & 2.00 & 2.03e-05 & 2.00 & 3.18e-05 & 2.00 \\
\hline
\multirow{6}{*}{1}
& 16 & 3.47e-02 & --- & 3.02e-03 & --- & 5.96e-03 & --- & 7.89e-03 & --- \\
& 64 & 8.86e-03 & 1.97 & 5.58e-04 & 2.44 & 8.72e-04 & 2.77 & 9.53e-04 & 3.05 \\
& 256 & 2.22e-03 & 1.99 & 7.92e-05 & 2.82 & 1.16e-04 & 2.92 & 1.18e-04 & 3.01 \\
& 1024 & 5.56e-04 & 2.00 & 1.02e-05 & 2.95 & 1.47e-05 & 2.98 & 1.48e-05 & 3.00 \\
& 4096 & 1.39e-04 & 2.00 & 1.29e-06 & 2.99 & 1.84e-06 & 2.99 & 1.84e-06 & 3.00 \\
& 16384 & 3.48e-05 & 2.00 & 1.62e-07 & 3.00 & 2.31e-07 & 3.00 & 2.30e-07 & 3.00 \\
\hline
\multirow{5}{*}{2}
& 16 & 4.51e-03 & --- & 2.55e-04 & --- & 3.51e-04 & --- & 6.14e-04 & --- \\
& 64 & 5.74e-04 & 2.98 & 1.30e-05 & 4.29 & 1.98e-05 & 4.14 & 4.18e-05 & 3.88 \\
& 256 & 7.20e-05 & 2.99 & 7.67e-07 & 4.09 & 1.21e-06 & 4.04 & 2.68e-06 & 3.96 \\
& 1024 & 9.01e-06 & 3.00 & 4.72e-08 & 4.02 & 7.50e-08 & 4.01 & 1.68e-07 & 3.99 \\
& 4096 & 1.13e-06 & 3.00 & 2.97e-09 & 3.99 & 4.69e-09 & 4.00 & 1.05e-08 & 4.00 \\
\hline
\multirow{4}{*}{3}
& 16 & 2.20e-04 & --- & 2.08e-05 & --- & 2.01e-05 & --- & 5.48e-05 & --- \\
& 64 & 1.39e-05 & 3.98 & 8.34e-07 & 4.64 & 8.38e-07 & 4.58 & 1.67e-06 & 5.03 \\
& 256 & 8.70e-07 & 4.00 & 2.82e-08 & 4.89 & 2.86e-08 & 4.87 & 5.18e-08 & 5.01 \\
& 1024 & 5.44e-08 & 4.00 & 9.08e-10 & 4.96 & 9.24e-10 & 4.95 & 1.62e-09 & 5.00 \\
\hline
\end{tabular}

  \medskip
  \caption{Errors and rates for the problem from Section~\ref{sec:exp1} with test norm
  $\norm\cdot{V}=\VnormQopt\cdot$.}
  \label{tab:exp1:opt}
\end{table}

\begin{table}[htb]
  \centering
  \begin{tabular}{|c|c|cc|cc|cc|cc|}
\hline
$p$ & $\#\TT$ & $\norm{u-u_h}{}$ & rate & $\norm{\Pi^p u-u_h}{}$ & rate & $\norm{u-u_h^+}{}$ & rate & $\norm{u-\widetilde u_h}{}$ & rate
\\ \hline\hline
\multirow{7}{*}{0}
& 16 & 1.92e-01 & --- & 6.88e-02 & --- & 7.86e-02 & --- & 8.48e-02 & --- \\
& 64 & 9.35e-02 & 1.04 & 1.73e-02 & 1.99 & 1.97e-02 & 1.99 & 2.17e-02 & 1.97 \\
& 256 & 4.64e-02 & 1.01 & 4.33e-03 & 2.00 & 4.94e-03 & 2.00 & 5.44e-03 & 1.99 \\
& 1024 & 2.32e-02 & 1.00 & 1.08e-03 & 2.00 & 1.23e-03 & 2.00 & 1.36e-03 & 2.00 \\
& 4096 & 1.16e-02 & 1.00 & 2.71e-04 & 2.00 & 3.09e-04 & 2.00 & 3.41e-04 & 2.00 \\
& 16384 & 5.79e-03 & 1.00 & 6.77e-05 & 2.00 & 7.71e-05 & 2.00 & 8.51e-05 & 2.00 \\
& 65536 & 2.89e-03 & 1.00 & 1.69e-05 & 2.00 & 1.93e-05 & 2.00 & 2.13e-05 & 2.00 \\
\hline
\multirow{6}{*}{1}
& 16 & 3.49e-02 & --- & 4.81e-03 & --- & 6.96e-03 & --- & 6.79e-03 & --- \\
& 64 & 8.87e-03 & 1.98 & 7.36e-04 & 2.71 & 9.71e-04 & 2.84 & 8.82e-04 & 2.95 \\
& 256 & 2.22e-03 & 2.00 & 9.82e-05 & 2.91 & 1.26e-04 & 2.95 & 1.12e-04 & 2.98 \\
& 1024 & 5.56e-04 & 2.00 & 1.25e-05 & 2.97 & 1.59e-05 & 2.99 & 1.41e-05 & 2.99 \\
& 4096 & 1.39e-04 & 2.00 & 1.57e-06 & 2.99 & 1.99e-06 & 3.00 & 1.76e-06 & 3.00 \\
& 16384 & 3.48e-05 & 2.00 & 1.96e-07 & 3.00 & 2.49e-07 & 3.00 & 2.20e-07 & 3.00 \\
\hline
\multirow{5}{*}{2}
& 16 & 4.53e-03 & --- & 4.38e-04 & --- & 5.07e-04 & --- & 5.22e-04 & --- \\
& 64 & 5.74e-04 & 2.98 & 2.53e-05 & 4.11 & 3.00e-05 & 4.08 & 3.25e-05 & 4.01 \\
& 256 & 7.20e-05 & 2.99 & 1.54e-06 & 4.04 & 1.85e-06 & 4.02 & 2.03e-06 & 4.00 \\
& 1024 & 9.01e-06 & 3.00 & 9.58e-08 & 4.01 & 1.15e-07 & 4.01 & 1.27e-07 & 4.00 \\
& 4096 & 1.13e-06 & 3.00 & 6.03e-09 & 3.99 & 7.22e-09 & 3.99 & 7.94e-09 & 4.00 \\
\hline
\multirow{4}{*}{3}
& 16 & 2.25e-04 & --- & 5.14e-05 & --- & 5.06e-05 & --- & 6.01e-05 & --- \\
& 64 & 1.40e-05 & 4.01 & 1.75e-06 & 4.88 & 1.73e-06 & 4.87 & 1.96e-06 & 4.94 \\
& 256 & 8.71e-07 & 4.00 & 5.62e-08 & 4.96 & 5.55e-08 & 4.96 & 6.20e-08 & 4.98 \\
& 1024 & 5.44e-08 & 4.00 & 1.80e-09 & 4.96 & 1.78e-09 & 4.96 & 1.96e-09 & 4.98 \\
\hline
\end{tabular}

  \medskip
  \caption{Errors and rates for the problem from Section~\ref{sec:exp1} with test norm
  $\norm\cdot{V}=\VnormSimple\cdot$.}
  \label{tab:exp1:std}
\end{table}

\begin{table}[htb]
  \centering
  \begin{tabular}{|c|c|cc|cc|cc|cc|}
\hline
$p$ & $\#\TT$ & $\norm{u-u_h}{}$ & rate & $\norm{\Pi^p u-u_h}{}$ & rate & $\norm{u-u_h^+}{}$ & rate & $\norm{u-\widetilde u_h}{}$ & rate
\\ \hline\hline
\multirow{7}{*}{0}
& 16 & 1.96e-01 & --- & 7.95e-02 & --- & 8.85e-02 & --- & 1.27e-01 & --- \\
& 64 & 9.41e-02 & 1.06 & 2.04e-02 & 1.96 & 2.25e-02 & 1.98 & 3.36e-02 & 1.92 \\
& 256 & 4.65e-02 & 1.02 & 5.14e-03 & 1.99 & 5.64e-03 & 1.99 & 8.51e-03 & 1.98 \\
& 1024 & 2.32e-02 & 1.00 & 1.29e-03 & 2.00 & 1.41e-03 & 2.00 & 2.13e-03 & 1.99 \\
& 4096 & 1.16e-02 & 1.00 & 3.22e-04 & 2.00 & 3.53e-04 & 2.00 & 5.34e-04 & 2.00 \\
& 16384 & 5.79e-03 & 1.00 & 8.05e-05 & 2.00 & 8.82e-05 & 2.00 & 1.34e-04 & 2.00 \\
& 65536 & 2.89e-03 & 1.00 & 2.01e-05 & 2.00 & 2.21e-05 & 2.00 & 3.34e-05 & 2.00 \\
\hline
\multirow{6}{*}{1}
& 16 & 3.47e-02 & --- & 2.77e-03 & --- & 5.91e-03 & --- & 8.02e-03 & --- \\
& 64 & 8.85e-03 & 1.97 & 5.22e-04 & 2.40 & 8.59e-04 & 2.78 & 9.73e-04 & 3.04 \\
& 256 & 2.22e-03 & 1.99 & 7.47e-05 & 2.80 & 1.14e-04 & 2.92 & 1.21e-04 & 3.01 \\
& 1024 & 5.56e-04 & 2.00 & 9.69e-06 & 2.95 & 1.44e-05 & 2.98 & 1.51e-05 & 3.00 \\
& 4096 & 1.39e-04 & 2.00 & 1.22e-06 & 2.99 & 1.81e-06 & 2.99 & 1.89e-06 & 3.00 \\
& 16384 & 3.48e-05 & 2.00 & 1.53e-07 & 3.00 & 2.27e-07 & 3.00 & 2.36e-07 & 3.00 \\
\hline
\multirow{5}{*}{2}
& 16 & 4.51e-03 & --- & 2.37e-04 & --- & 3.44e-04 & --- & 6.25e-04 & --- \\
& 64 & 5.73e-04 & 2.98 & 1.19e-05 & 4.32 & 1.95e-05 & 4.14 & 4.24e-05 & 3.88 \\
& 256 & 7.20e-05 & 2.99 & 6.97e-07 & 4.09 & 1.19e-06 & 4.04 & 2.72e-06 & 3.97 \\
& 1024 & 9.01e-06 & 3.00 & 4.28e-08 & 4.02 & 7.37e-08 & 4.01 & 1.71e-07 & 3.99 \\
& 4096 & 1.13e-06 & 3.00 & 2.68e-09 & 4.00 & 4.60e-09 & 4.00 & 1.07e-08 & 4.00 \\
\hline
\multirow{4}{*}{3}
& 16 & 2.20e-04 & --- & 1.95e-05 & --- & 1.98e-05 & --- & 5.51e-05 & --- \\
& 64 & 1.39e-05 & 3.98 & 7.80e-07 & 4.64 & 8.14e-07 & 4.61 & 1.68e-06 & 5.04 \\
& 256 & 8.70e-07 & 4.00 & 2.65e-08 & 4.88 & 2.78e-08 & 4.87 & 5.21e-08 & 5.01 \\
& 1024 & 5.44e-08 & 4.00 & 8.73e-10 & 4.92 & 9.16e-10 & 4.92 & 1.63e-09 & 5.00 \\
\hline
\end{tabular}

  \medskip
  \caption{Errors and rates for the problem from Section~\ref{sec:exp2} with test norm
  $\norm\cdot{V}=\VnormQopt\cdot$.}
  \label{tab:exp2:opt}
\end{table}

\begin{table}[htb]
  \centering
  \begin{tabular}{|c|c|cc|cc|cc|cc|}
\hline
$p$ & $\#\TT$ & $\norm{u-u_h}{}$ & rate & $\norm{\Pi^p u-u_h}{}$ & rate & $\norm{u-u_h^+}{}$ & rate & $\norm{u-\widetilde u_h}{}$ & rate
\\ \hline\hline
\multirow{7}{*}{0}
& 16 & 4.37e-01 & --- & 3.98e-01 & --- & 4.15e-01 & --- & 4.00e-01 & --- \\
& 64 & 2.25e-01 & 0.96 & 2.06e-01 & 0.95 & 2.14e-01 & 0.96 & 2.06e-01 & 0.96 \\
& 256 & 1.14e-01 & 0.99 & 1.04e-01 & 0.98 & 1.08e-01 & 0.99 & 1.04e-01 & 0.99 \\
& 1024 & 5.70e-02 & 1.00 & 5.21e-02 & 1.00 & 5.41e-02 & 1.00 & 5.21e-02 & 1.00 \\
& 4096 & 2.85e-02 & 1.00 & 2.61e-02 & 1.00 & 2.71e-02 & 1.00 & 2.61e-02 & 1.00 \\
& 16384 & 1.43e-02 & 1.00 & 1.30e-02 & 1.00 & 1.35e-02 & 1.00 & 1.30e-02 & 1.00 \\
& 65536 & 7.13e-03 & 1.00 & 6.52e-03 & 1.00 & 6.77e-03 & 1.00 & 6.52e-03 & 1.00 \\
\hline
\multirow{6}{*}{1}
& 16 & 6.23e-02 & --- & 5.18e-02 & --- & 5.69e-02 & --- & 1.64e-02 & --- \\
& 64 & 1.63e-02 & 1.93 & 1.37e-02 & 1.92 & 1.49e-02 & 1.93 & 5.70e-03 & 1.52 \\
& 256 & 4.12e-03 & 1.98 & 3.47e-03 & 1.98 & 3.77e-03 & 1.98 & 1.61e-03 & 1.83 \\
& 1024 & 1.03e-03 & 2.00 & 8.67e-04 & 2.00 & 9.42e-04 & 2.00 & 4.19e-04 & 1.94 \\
& 4096 & 2.57e-04 & 2.00 & 2.17e-04 & 2.00 & 2.35e-04 & 2.00 & 1.06e-04 & 1.98 \\
& 16384 & 6.43e-05 & 2.00 & 5.41e-05 & 2.00 & 5.88e-05 & 2.00 & 2.68e-05 & 1.99 \\
\hline
\multirow{5}{*}{2}
& 16 & 7.46e-03 & --- & 5.95e-03 & --- & 6.56e-03 & --- & 9.34e-04 & --- \\
& 64 & 9.32e-04 & 3.00 & 7.35e-04 & 3.02 & 8.17e-04 & 3.00 & 6.23e-05 & 3.91 \\
& 256 & 1.17e-04 & 3.00 & 9.17e-05 & 3.00 & 1.02e-04 & 3.00 & 4.01e-06 & 3.96 \\
& 1024 & 1.46e-05 & 3.00 & 1.15e-05 & 3.00 & 1.28e-05 & 3.00 & 2.57e-07 & 3.96 \\
& 4096 & 1.82e-06 & 3.00 & 1.43e-06 & 3.00 & 1.59e-06 & 3.00 & 1.66e-08 & 3.95 \\
\hline
\multirow{4}{*}{3}
& 16 & 6.03e-04 & --- & 5.62e-04 & --- & 5.59e-04 & --- & 7.46e-05 & --- \\
& 64 & 3.88e-05 & 3.96 & 3.63e-05 & 3.95 & 3.64e-05 & 3.94 & 3.93e-06 & 4.25 \\
& 256 & 2.44e-06 & 3.99 & 2.28e-06 & 3.99 & 2.29e-06 & 3.99 & 2.41e-07 & 4.03 \\
& 1024 & 1.52e-07 & 4.00 & 1.42e-07 & 4.00 & 1.43e-07 & 4.00 & 1.52e-08 & 3.99 \\
\hline
\end{tabular}

  \medskip
  \caption{Errors and rates for the problem from Section~\ref{sec:exp2} with test norm
  $\norm\cdot{V}=\VnormSimple\cdot$.}
  \label{tab:exp2:std}
\end{table}

Define $T_1 := \operatorname{conv}\{(0,0),(1,0),(\tfrac12,\tfrac12)\}$, $T_2 
:=\operatorname{conv}\{(1,1),(0,1),(\tfrac12,\tfrac12)\}$.
In the first example we set $\bbeta = 0$ and
\begin{align*}
  \gamma(x,y) := \begin{cases}
    1 & (x,y) \in T_1, \\
    \tfrac12 & (x,y) \in T_2 , \\
    0 & (x,y)\in \Omega\setminus(T_1\cup T_2).
  \end{cases}
\end{align*}
Moreover, we choose 
\begin{align*}
  \ff(x,y):= \begin{cases}
    \begin{pmatrix}
      1 \\ 1
    \end{pmatrix}
    & x<\tfrac12, \\
    \begin{pmatrix}
      1 \\ -1
    \end{pmatrix}
    & x\geq\tfrac12.
  \end{cases}
\end{align*}
Note that $\div\ff = 0$ and $\ff\in \Hdivset\Omega\cap\HH^1(\TT_1)$.
With the coefficients, $\ff$ and the exact solution at hand, we calculate the right-hand side $f$ and $\ssigma$
through~\eqref{eq:model}.

Table~\ref{tab:exp1:opt} resp. Table~\ref{tab:exp1:std} show errors and convergence rates when
using the test norm $\VnormQopt\cdot$ resp. $\VnormStd\cdot=\VnormSimple\cdot$.
We observe higher convergence rates as predicted by our main results.

\subsection{Example 2}\label{sec:exp2}
For this example we choose $\ff = 0$, $\gamma = 0$, $\bbeta(x,y) = (1,1)^T$ for $(x,y)\in \Omega$. 
Note that $\bbeta$ is smooth.
Again we calculate $f$ and $\ssigma$ through~\eqref{eq:model}. 
Table~\ref{tab:exp2:opt} resp. Table~\ref{tab:exp2:std} show the results for Case~\ref{case:a} ($\norm\cdot{V} =
\VnormQopt\cdot$) resp. Case~\ref{case:b},~\ref{case:c} ($\norm\cdot{V}=\VnormStd\cdot=\VnormSimple\cdot$).
Observe from Table~\ref{tab:exp2:std} that we do not get higher convergence rates neither for solutions from the
augmented space $U_{hp}^+$ nor for the postprocessed solution. 
Even for the $L^2$ error of $\Pi^p u-u_h$ we do not get higher rates, whereas 
with the use of the quasi-optimal test norm $\VnormQopt\cdot$ higher rates are obtained.
This demonstrates that the assumption $\bbeta=0$ in Section~\ref{sec:ass} for the Cases~\ref{case:b}--\ref{case:c}
is not an artefact used in the proofs but in general is also necessary to obtain superconvergence results with the norms
$\VnormStd\cdot$, $\VnormSimple\cdot$.

\section{Concluding remarks}\label{sec:remarks}
We conclude this work with some remarks.
The results and their proofs are presented in a systematic way that allow 
to extend and transfer them to other types of meshes and different model problems.
In principle, the crucial results Lemma~\ref{lem:w:rep} and Lemma~\ref{lem:est} have to be verified.
Consider for instance that $\TT$ is a mesh with polygonal elements.
Lemma~\ref{lem:w:rep} still holds true in that case since it is independent of the underlying mesh so that 
only the assertion of Lemma~\ref{lem:est} has to be shown.
To be more precise: Analyzing the proof one finds out that it only remains to provide the estimate
\begin{align*}
  \min_{\ww_h\in U_h}\norm{\ww-\ww_h}U + \min_{\vv_k\in V_{hk}}\norm{\vv-\vv_k}V \lesssim h \norm{g}{},
\end{align*}
which is an optimal a priori error bound for sufficient regular functions (see Lemma~\ref{lem:est} for details on the definition of
the functions $\ww$ and $\vv$).
In the case of triangular meshes we have proven the estimate by using basic properties of well-known interpolation
operators. 
If operators with the same properties can be defined on meshes with polygonal elements, then, clearly, the estimate
holds true as well.
We note that the analysis of DPG methods for ultra-weak formulations on general (polygonal) meshes is an ongoing research.
For an overview we refer to the recent work~\cite{PolyDPG}.

Future research will include other model problems, e.g., linear elasticity.
Another possible application of the developed ideas could be to the Stokes problem. Consider its
velocity-gradient-pressure formulation: Find $(\uu_S,\ssigma_S,p_S)$ such that
\begin{alignat*}{2}
  -\nabla p_S+\div\ssigma_S &= \ff &\quad&\text{in }\Omega, \\
  \ssigma_S-\nabla\uu_S &= 0 &\quad&\text{in }\Omega, \\
  \div\uu_S &= 0 &\quad&\text{in }\Omega, \\
  \uu_S &= 0 &\quad&\text{on }\partial\Omega.
\end{alignat*}
DPG methods based on ultra-weak formulations are known and thoroughly analyzed~\cite{DPGStokes}. 
Since regularity theory is also known, our main results (Theorem~\ref{thm:L2}--\ref{thm:postproc}) should carry over (for the
velocity variable $\uu_S$ instead of $u$) to the Stokes problem following the same lines in the proofs.
In particular, the assertion of Theorem~\ref{thm:L2} has been already observed in numerical
experiments~\cite[Section~3]{DPGStokes} even for different test norms.
We refer also to~\cite[Section~3]{DPGNavierStokes} for numerical evidence in the case of incompressible Navier Stokes
problems.

Another point we like to mention is that the principal ideas of the proofs and, thus, our main results carry
over to the low regularity case, i.e., when we do not have the ``full'' regularity $u\in H^2(\Omega)$, $v\in H^2(\Omega)$
for solutions of~\eqref{eq:model} and~\eqref{eq:adjoint} but rather $u\in H^{1+s}(\Omega)$, $v\in H^{1+s}(\Omega)$ for
some $s\in(\tfrac12,1)$. This is usually the case when $\Omega$ is a nonconvex polygonal domain.
Nevertheless, we stress that our main results (Theorem~\ref{thm:L2}--\ref{thm:postproc})
hold true with $h^{p+2}$ replaced by $h^{p+1+s}$. Therefore, one still obtains higher convergence rates than the overall
error $\norm{\uu-\uu_h}{}=\OO(h^{p+1})$.
For the particular case of a reaction-diffusion model problem ($\CC$ is the identity, $\bbeta=0$, and $\gamma=1$)
Theorem~\ref{thm:augment} and~\ref{thm:postproc} are analyzed in~\cite{SupconvDPG} for $\norm\cdot{V} =
\VnormStd\cdot=\VnormSimple\cdot$.

Finally, let us remark the importance of the choice of norms in the test space. 
Although all test norms under consideration are equivalent and, thus, the
corresponding DPG methods have the same stability properties (i.e., the $\inf$--$\sup$ constants resp. boundedness
constants are equivalent), only one of the norms under consideration (the quasi-optimal norm $\VnormQopt\cdot$) 
yields higher convergence rates for general model problems with $\bbeta\neq 0$.
This has to be taken into account in the design of DPG methods.

\bibliographystyle{abbrv}
\bibliography{literature}

\end{document}